\documentclass[a4paper, 10pt]{amsart}

\usepackage{arydshln}
\usepackage[latin1]{inputenc}
\usepackage[shortlabels]{enumitem}
\usepackage{nameref}
\usepackage{hyperref}
\usepackage[nospace,noadjust]{cite}
 \usepackage[T1]{fontenc}
\usepackage{lastpage, setspace}
\usepackage{pst-all}
\usepackage[all]{xy}
\usepackage{tikz}
\usepackage{amsmath, amsthm, amssymb, tensor}
\usepackage{listings}
\usepackage{verbatim}
\usepackage{amsmath,bm}

\usepackage{pdflscape}
\usepackage{filecontents}
\usepackage[foot]{amsaddr}
\usepackage{float}

\makeatletter
\renewcommand{\email}[2][]{%
  \ifx\emails\@empty\relax\else{\g@addto@macro\emails{,\space}}\fi%
  \@ifnotempty{#1}{\g@addto@macro\emails{\textrm{(#1)}\space}}%
  \g@addto@macro\emails{#2}%
}
\makeatother

\makeatletter
\newcommand{\displaybump}{\hbox to \@totalleftmargin{\hfil}}
\makeatother

\SelectTips{cm}{}
\newdir{ >}{{}*!/-5pt/\dir{>}}

\newcommand{\Conv}{\operatorname{\scalebox{1.7}{\raisebox{-0.3ex}{$\ast$}}}}%
\newcommand{\Conn}{\operatorname{\scalebox{1.7}{\raisebox{-0.3ex}{$\#$}}}}%

\DeclareMathOperator\Aut{Aut}

\DeclareMathOperator\id{id}

\DeclareMathOperator\image{im}

\DeclareMathOperator\bbN{\mathbb{N}}

\DeclareMathOperator\bbZ{\mathbb{Z}}

\DeclareMathOperator\calS{\mathcal{S}}

\newtheorem{theorem}{Theorem}[section]
\newtheorem*{theorem*}{Theorem}
\newtheorem{lemma}[theorem]{Lemma}
\newtheorem*{lemma*}{Lemma}
\newtheorem{corollary}[theorem]{Corollary}
\newtheorem{proposition}[theorem]{Proposition}
\newtheorem*{proposition*}{Proposition}

\theoremstyle{definition}
\newtheorem{definition}[theorem]{Definition}
\newtheorem*{definition*}{Definition}
\newtheorem{remark}[theorem]{Remark}
\newtheorem*{remark*}{Remark}
\newtheorem{example}[theorem]{Example}
\newtheorem*{example*}{Example}
\newtheorem*{sketch of proof}{Sketch of Proof}
\newtheorem*{idea of proof}{Idea of Proof}
\newtheorem{question}[theorem]{Question}

\title[On free products of graphs]{On Free products of graphs}
\author{Max Carter}
\email{max.carter@newcastle.edu.au}
\author{Stephan Tornier}
\email{stephan.tornier@newcastle.edu.au}
\author{George Willis}
\email{george.willis@newcastle.edu.au}
\address[All authors]{The University of Newcastle, School of Mathematical and Physical Sciences, 2308 Callaghan NSW, Australia.}
\date{\today}

\begin{document}

\begin{abstract}
We define a free product of connected simple graphs that is equivalent to several existing definitions when the graphs are vertex-transitive but differs otherwise. The new definition is designed for the automorphism group of the free product to be as large as possible, and we give sufficient criteria for it to be non-discrete. Finally, we transfer Tits' classification of automorphisms of trees and simplicity criterion to free products of graphs.
\end{abstract}

\maketitle

\section{Introduction}

Every locally compact (l.c.) group $G$ is an extension of its connected component, $G_{0}$, by the totally disconnected (t.d.) quotient $G/G_{0}$:
\begin{displaymath}
 \xymatrix{
  1 \ar[r] & G_{0} \ar[r] & G \ar[r] & G/G_{0} \ar[r] & 1.
 }
\end{displaymath}
As connected l.c. groups have been seen to be inverse limits of Lie groups in fundamental work by Gleason \cite{Gle52}, Montgomery-Zippin \cite{MZ52} and Yamabe \cite{Yam53}, totally disconnected ones are now the focus in the structure theory of l.c.~groups.

Every t.d.l.c. group can be viewed as a directed union of compactly generated (c.g.) open subgroups. Among c.g.t.d.l.c. groups, those acting on (regular) graphs stand out due to the Cayley-Abels graph construction \cite{KM08}: Every c.g.t.d.l.c. group $G$ acts vertex-transitively on a connected regular graph $\Gamma$ of finite degree with compact open vertex-stabilisers. Conversely, the automorphism group of every connected locally finite graph is t.d.l.c. and has compact open vertex-stabilisers, see Section~\ref{sec:preliminaries}.

Groups acting on graphs are studied extensively \cite{BM00a} \cite{BEW15} \cite{Tor20} and play a significant role when studying the scale of a t.d.l.c.~group \cite{Wil94}, see {\it e.g.\/} \cite{BW04}, \cite{Moe02} and \cite{BT19}. It remains a frequently appearing and difficult question to determine when the automorphism group of a given (infinite) graph is non-discrete.

\vspace{0.2cm}
In this article, we define a free product $\Gamma=\Conv_{i=1}^{n}\Gamma_{i}$ of connected simple graphs $\Gamma_{1},\ldots,\Gamma_{n}$ $(n\in\bbN)$, which may be (locally) infinite. This product is associative and commutative with unit a single vertex. It realises the Cayley graph of a free product of finitely generated groups as the free product of the Cayley graphs of the factors. We compare our definition with several existing ones and see that all are equivalent when $\Gamma_{1},\ldots,\Gamma_{n}$ are vertex-transitive. The vertices and edges of $\Gamma$ are defined directly, rather than by gluing copies of $\Gamma_{1},\ldots,\Gamma_{n}$ as done in the other definitions, which facilitates computation in the graphs and also, when the graphs are not vertex-transitive, yields a free product with a larger automorphism group. We give sufficient criteria for $\Aut(\Gamma)$ to be non-discrete: Corollary~\ref{cor:aut_non_discrete_stabiliser} states that $\Aut(\Gamma)$ is non-discrete as soon as one of $\Gamma_{1},\ldots,\Gamma_{n}$ $(n\ge 2)$ has some non-trivial vertex-stabiliser, and Proposition~\ref{prop:aut_non_discrete_isomorphic} states that $\Aut(\Gamma)$ is non-discrete whenever two of the factors $\Gamma_{1},\ldots,\Gamma_{n}$ $(n\ge 3)$ are isomorphic.

Finally, we consider the subgroup $\Aut_{\calS}(\Gamma)$ of $\Aut(\Gamma)$ which preserves the sheet structure of $\Gamma$ induced by its factors. It acts on a \emph{structure tree}, akin to a block-cut tree, and thus lends itself for the transfer of Tits' classification of automorphisms of trees and simplicity theorem \cite{Tit70}, see Theorem~\ref{thm:aut_s_three_types} and Proposition~\ref{prop:aut_s_simplicity}.

\section{Preliminaries}\label{sec:preliminaries}

A \emph{graph} $\Gamma$ is a tuple $(V,E)$ consisting of a \emph{vertex set} $V$ and an \emph{edge set} $E\subseteq V^{|2|}$, where $V^{|2|}:=\{\{x,y\}\mid x,y\in V,\ x\neq y\}$. The \emph{order} of $\Gamma$ is $|V|$. For $x\in V$, we let $N_{\Gamma}(x)$ denote the set $\{y\in V\mid \{x,y\}\in E\}$ of \emph{neighbours of $x\in V$} and $E_{\Gamma}(x)$ the set $\{e\in E\mid x\in e\}$ of edges incident with $x\in V$. The \emph{valency}, or \emph{degree}, of $x\in V$ is $|N_{\Gamma}(x)|$. The graph $\Gamma$ is \emph{locally finite} if $|N_{\Gamma}(x)|$ is finite for every $x\in V$. It is \emph{regular of degree $d$} if $|N_{\Gamma}(x)|=d$ for every $x\in V$.

A \emph{path of length~$k$}, where $k\in\bbN$, in $\Gamma$ is a sequence $(e_1,\dots, e_k)$ of edges with $e_i\cap e_{i+1} = \{x_i\}$ a singleton for each $i\in \{1,\dots, k-1\}$. The path may thus also be described by the sequence $(x_0,\dots, x_{k})$ of vertices, where $e_{i}=\{x_{i-1},x_i\}$ for $i\in\{1,\ldots,k\}$, and $x_{i}\neq x_{i+2}$ for all $i\in\{0,\ldots,k-2\}$. The graph $\Gamma$ is \emph{connected} if for all $x,y\in V$ there is a path $(x=x_{0},x_{1},\ldots,x_{k}=y)$ from $x$ to $y$. It is a \emph{forest} if it does not contain paths from a vertex to itself. A \emph{tree} is a connected forest. We denote the regular tree of degree $d\in\bbN$ by $T_{d}$. Note that $T_{1}$ consists of a single edge. Given a subset $S\subseteq V$, the \emph{span of $S$ in $\Gamma$} is the graph $\mathrm{span}_{\Gamma}(S):=(S,E\cap S^{|2|})$.

Let $\Gamma=(V,E)$ be a connected graph. A vertex $x\in V$ is a \emph{cut vertex} if the graph $(V\backslash\{x\},E\backslash E_{\Gamma}(x))$, which arises from $\Gamma$ by removing $x\in V$ and all edges incident with $v$, is not connected. A graph without cut-vertices is \emph{$2$-connected}.

Existing definitions of free products involve repeatedly joining together copies of the graphs at specified, or root, vertices. Rooted graphs $\Gamma_{1}=(V_{1},E_{1},v_{1})$ and $\Gamma_{2}=(V_{2},E_{2},v_{2})$ are joined at $v_1$ and $v_2$ to obtain the \emph{connected sum} $(\Gamma_{1},v_{1})\#(\Gamma_{2},v_{2})$, which is the graph obtained from $\Gamma_{1}\sqcup\Gamma_{2}=(V_{1}\sqcup V_{2},E_{1}\sqcup E_{2})$ by identifying $v_{1}$ and $v_{2}$.

Let $\Gamma_{1}=(V_{1},E_{1})$ and $\Gamma_{2}=(V_{2},E_{2})$ be graphs. A \emph{morphism} $\varphi:\Gamma_{1}\to\Gamma_{2}$ is a map $\varphi:V_{1}\to V_{2}$ such that $\varphi(\{x,y\})\in E_{2}$ for every edge $\{x,y\}\in E_{1}$.

\begin{lemma}\label{lem:quotientmap}
Let $\Gamma\!=\!(V,E)$ and $\Gamma'\!=\!(V',E')$ be graphs and $\varphi:\!\Gamma\to\Gamma'$ a morphism. If $\Gamma'$ is connected and $\varphi(N_{\Gamma}(x))=N_{\Gamma'}(\varphi(x))$ for every $x\in V$ then $\varphi$ is surjective.
\end{lemma}

\begin{proof}
Let $x\in V$ and $x_{0}:=\varphi(x)\in V'$. Given $y\in V'$ let $(x_{0},x_{1},\ldots,x_{k}=y)$ be a path from $x_{0}$ to $y$. Since $\varphi(N_{\Gamma}(x))=N_{\Gamma'}(\varphi(x))=N_{\Gamma'}(x_{0})$ we conclude that $x_{1}\in\varphi(N_{\Gamma}(x))$, say $x_{1}=\varphi(x')$. Iterate.
\end{proof}

Let $\Gamma=(V,E)$ be a connected graph. We equip the set $\Aut(\Gamma)$ of automorphisms of $\Gamma$ with the permutation topology for its action on $V$, see \cite[Section~1.2]{Moe02}. This turns $\Aut(\Gamma)$ into a Hausdorff totally disconnected group that is locally compact when $\Gamma$ is locally finite. Given a subgroup $H\le\Aut(\Gamma)$ and a subgraph $\Gamma'\subseteq\Gamma$, the \emph{pointwise stabiliser} of $\Gamma'$ in $H$ is denoted by $H_{\Gamma'}$. Similarly, the \emph{setwise stabiliser} of $\Gamma'$ in $H$ is $H_{\{\Gamma'\}}$. Finally, the \emph{rigid stabiliser} of $\Gamma'$ in $H$ is the set $\mathrm{rist}_{H}(\Gamma'):=H_{\Gamma\backslash\Gamma'}$ of elements which pointwise stabilise the complement of $\Gamma'$ in $\Gamma$.

\section{Definition}\label{sec:def}

In Sections \ref{sec:def_mswz}, \ref{sec:def_pt} and \ref{sec:def_q} below we recall several existing definitions of free products of graphs found in the literature, all of which are in the spirit of \cite[Section 31.2]{HIK11}. The new definition given in Section \ref{sec:def_new} has the advantage that it lends itself to computations and the analysis of the free product's automorphism group. A comparison between these four definitions made in Section \ref{sec:equivalence} illuminates all of them and shows that they are equivalent in the case of vertex-transitive graphs.

\subsection{A new definition}\label{sec:def_new}

Let $\Gamma_{1}=(V_{1},E_{1}),\ldots,\Gamma_{n}=(V_{n},E_{n})$ be connected graphs. The \emph{free product} of $\Gamma_{1},\ldots,\Gamma_{n}$, denoted $\Conv_{i=1}^{n}\Gamma_{i}$, is the following graph $\Gamma=(V,E)$. Its vertex set consists of \emph{admissible} words in the alphabet $V_{1}\sqcup\cdots\sqcup V_{n}$ and its edge set is defined in terms of the last vertices in such admissible words.

\subsubsection*{Vertices} The set of admissible words of length~$l$ is denoted by $V^{(l)}$ and defined by induction on~$l$. Its definition involves an \emph{update function} $\mathbf{u}:V\to V_1\times\dots\times V_n$ defined inductively alongside $V^{(l)}$. Choose $(u_{1},\ldots,u_{n})\in V_{1}\times\dots\times V_{n}$ and set
\begin{equation}
\label{eq:initial}
V^{(0)}:=\{\emptyset\}\mbox{ and }\mathbf{u}(\emptyset):=(u_1,\dots, u_n)\in V_1\times \dots \times V_n.
\end{equation}
Assume that $V^{(k)}$ and $\mathbf{u}|_{V^{(k)}}$ are defined for $k\in\{0,\dots, l\}$. For $\tilde{v}=v_{1}\ldots v_{k}\in V^{(k)}$, let $\mathbf{u}(\tilde{v}) = \left(u_1(\tilde{v}), \dots, u_n(\tilde{v})\right)$. Then $V^{(l+1)}$ consists of all words of the form $\tilde{v}v_{l+1}$ where $\tilde{v}=v_{1}\ldots v_{l}\in V^{(l)}$ and, supposing that $v_l\in V_i$ and $v_{l+1} \in V_j$, we have 
\begin{equation}
\label{eq:admissible}
j\ne i  \mbox{ and } v_{l+1} \ne u_j(\tilde{v}).
\end{equation}
The update function is defined on $V^{(l+1)}$ by 
\begin{equation}
\label{def:update}
u_h(\tilde{v}v_{l+1}) := \begin{cases}
u_h(\tilde{v}) & \mbox{ if } h\ne j \\
v_{l+1} & \mbox{ if }h=j
\end{cases}.
\end{equation}

\subsubsection*{Edges} The edge set $E$ consists of pairs of admissible words $\{\tilde{v},\tilde{v}v_{l+1}\}$ as follows. The edges incident on the empty word are 
\begin{equation}\label{eq:emptyneighbours}
\{\emptyset, v_1\} \mbox{ where }\{u_i(\emptyset),v_1\}\in E_i \mbox{ for some } i\in\{1,\dots, n\}.
\end{equation}
The edges incident on the non-empty word $\tilde{v} = v_1\dots v_{l-1}v_l$, with $v_l\in V_i$, are 
\begin{align}
\label{eq:nonemptyneighbours1}
& \{v_1\dots v_{l-1}v_l, v_1\dots v_{l-1}v'_l\} \mbox{ for all }\{v_l,v_l'\}\in E_i\\
\label{eq:nonemptyneighbours2}
\mbox{ and }\quad &\{\tilde{v}, \tilde{v}v_{l+1}\} \mbox{ for all } \{u_j(\tilde{v}),v_{l+1}\}\in E_j\mbox{ and } j\neq i.
\end{align}

\vspace{0.2cm}
\begin{remark}\label{rem:def_obs}
While the definition above depends on the choice of~$\mathbf{u}(\emptyset)$, it will be seen in Corollary~\ref{cor:uniqueness} that two graphs defined using different choices are isomorphic. The following observations further illuminate the definition of $\Gamma$.
\begin{enumerate}[(i),series=def_remarks]
 \item By Definition~\ref{sec:def_new}~\eqref{eq:admissible}, we have $v_1\dots v_l\!\in\! V$ if and only if $v_{k}$ and $v_{k+1}$ are vertices in distinct graphs for each $k\!\in\!\{1,\dots,l-1\}$, and $v_{k+1}\!\neq\! u_j(v_1\dots v_k)$, assuming $v_{k+1}\in V(\Gamma_{j})$.
 \item\label{item:valency} By Definition~\ref{sec:def_new}~\eqref{eq:emptyneighbours}, \eqref{eq:nonemptyneighbours1} and \eqref{eq:nonemptyneighbours2} the vertex $\tilde{v}\in V$ has one neighbour for each neighbour of $u_i(\tilde{v})$ ($i\in\{1,\dots, n\}$). Hence $|N_\Gamma(\tilde{v})|=\sum_{i=1}^n|N_{\Gamma_i}(u_i(\tilde{v}))|$.
\end{enumerate}
Definitions~\ref{sec:def_new}~\eqref{eq:nonemptyneighbours1} and~\eqref{eq:nonemptyneighbours2} imply that, for $\tilde{v}=v_{1}\ldots v_{l}\in V$ with $v_l\in V_i$ and $j\neq i$, the set of vertices $\left\{\tilde{v}\right\} \cup \left\{\tilde{v}v_{l+1}\mid v_{l+1}\in V_j\setminus\{ u_j(\tilde{v})\}\right\}$ spans a subgraph $S_{(\tilde{v},j)}$ of $\Gamma$ isomorphic to $\Gamma_j$, termed the \emph{$\Gamma_j$-sheet at $\tilde{v}$}. Similarly, the subgraph $S_{(\tilde{v},i)}$ spanned by $\{v_1\dots v_{l-1}\}\cup\left\{v_1\dots v_{l-1}v'_{l}\mid  v'_l\in V_i\setminus\{ u_i(v_1\dots v_{l-1})\}\right\}$ is the \emph{$\Gamma_{i}$-sheet at $\tilde{v}$}. Elements of the set $\calS:=\{S_{(\tilde{v},i)}\mid \tilde{v}\in V(\Gamma),\ i\in\{1,\ldots,n\}\}$ are \emph{sheets} of $\Gamma$.
\begin{enumerate}[(i),resume=def_remarks]
\item Every vertex of $\Gamma$ belongs to a unique $\Gamma_j$-sheet for every $j\in\{1,\ldots,n\}$. Belonging to the same $\Gamma_j$-sheet is an equivalence relation on $V$.
\item\label{item:sheet_colouring} Every edge of $\Gamma$ belongs to a unique $\Gamma_j$-sheet for a unique $j\in\{1,\ldots,n\}$. The associated map $c:E\to\{1,\dots,n\}$ defines a colouring of $\Gamma$ which is constant on sheets. Part \ref{item:valency} states that the number of edges adjacent to $\tilde{v}\in V$ with colour~$i$ is $|N_{\Gamma_i}(\tilde{v})|$. 
\item Sheets of different graphs intersect either trivially or at a single vertex.
\item If $\tilde{v},\tilde{w}\in V$ belong to the same $\Gamma_i$-sheet, then $u_j(\tilde{v}) = u_j(\tilde{w})$ for all $j\ne i$.
\end{enumerate}
\end{remark}

\begin{example}
This example illustrates the genesis of the free product $\Gamma$ of the graphs in Figure~\ref{fig:gen1}.

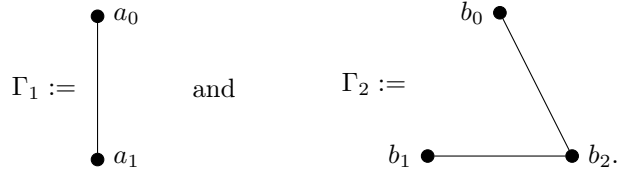
\begin{figure}[H]
\begin{center}
\begin{tikzpicture}[scale=0.95]
\node (G) at (-1.5,1) [shape=circle,label=0:{\normalsize $\Gamma_{1}:=$}] {};
\node (3) at (0,2) [shape=circle,draw, fill=black!100,scale=0.5,label={[label distance=0cm]0:{\normalsize $a_0$}}] {};
\node (4) at (0,0) [shape=circle,draw, fill=black!100,scale=0.5,label={[label distance=0cm]0:{\normalsize $a_1$}}] {};
\node (A) at (1,1) [shape=circle,label=0:{\normalsize $\text{and}$}] {};
\draw[-] (3) -- (4);
\end{tikzpicture}
\qquad
\begin{tikzpicture}[scale=0.95]
\node (G) at (-2.5,1) [shape=circle,label=0:{\normalsize $\Gamma_{2}:=$}] {};
\node (0) at (0,2) [shape=circle,draw, fill=black!100,scale=0.5,label=180:$b_{0}$] {};
\node (2) at (1,0) [shape=circle,draw, fill=black!100,scale=0.5,label={[label distance=0cm]0:{\normalsize $b_2.$}}] {};
\node (1) at (-1,0) [shape=circle,draw, fill=black!100,scale=0.5,label={[label distance=0cm]180:{\normalsize $b_1$}}] {};
\draw[-] (1) -- (2);
\draw[-] (2) -- (0);
\end{tikzpicture}
\end{center}
\caption{The graphs $\Gamma_{1}$ and $\Gamma_{2}$.}
\label{fig:gen1}
\end{figure}

\noindent
Starting from $\{\emptyset\}=V^{(0)}\subset V(\Gamma)$ and $\mathbf{u}(\emptyset)=(a_{0},b_{0})$, we obtain Figure~\ref{fig:gen2}
\begin{figure}[H]
\begin{center}
\begin{tikzpicture}[scale=0.95]
\node (0) at (-1,-1) [shape=circle,draw,scale=0.5,fill=black!100,label={[label distance=0cm]90:{\normalsize $b_1$}}] {};
\node (1) at (1,-1) [shape=circle,draw,scale=0.5,fill=black!100,label={[label distance=0cm]45:{\normalsize $b_2$}}] {};
\node (2) at (0,1) [shape=circle,scale=0.8] {$\emptyset$};
\node (3) at (0,3) [shape=circle,draw,scale=0.5,fill=black!100,label={[label distance=0cm]225:{\normalsize $a_1$}}] {};
\node (A) at (2,1) [shape=circle,label=0:{\normalsize $\text{and}$}] {};
\draw[-] (1) -- (2);
\draw[-] (2) -- (3);
\end{tikzpicture}
\quad
\begin{tikzpicture}[scale=0.95]
\node (0) at (-1,-1) [shape=circle,draw,scale=0.5,fill=black!100,label={[label distance=0cm]90:{\normalsize $b_1$}}] {};
\node (1) at (1,-1) [shape=circle,draw,scale=0.5,fill=black!100,label={[label distance=0cm]45:{\normalsize $b_2$}}] {};
\node (2) at (0,1) [shape=circle,scale=0.8] {$\emptyset$};
\node (3) at (0,3) [shape=circle,draw,scale=0.5,fill=black!100,label={[label distance=0cm]225:{\normalsize $a_1$}}] {};
\draw[-] (0) -- (1);
\draw[-] (1) -- (2);
\draw[-] (2) -- (3);
\end{tikzpicture}
\end{center}
\caption{The first two steps in the genesis of $\Gamma_{1}\ast\Gamma_{2}$.}
\label{fig:gen2}
\end{figure}
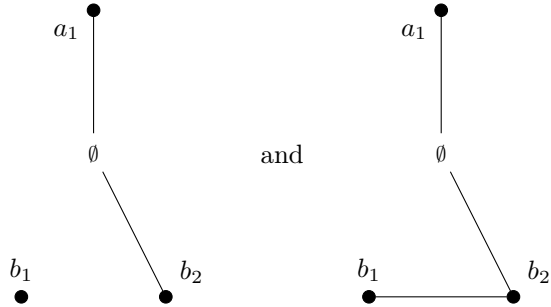

\noindent
by first adding in $V^{(1)}$ and edges incident on $\emptyset\in V(\Gamma)$, and then edges of type \eqref{eq:nonemptyneighbours1}. Here, $\textbf{u}(a_1)=(\textcolor{blue}{a_1}, b_0)$, $\textbf{u}(b_1)=(a_0,\textcolor{blue}{b_1})$ and $\textbf{u}(b_2)=(a_0,\textcolor{blue}{b_2})$. After adding in $V^{(2)}$ and edges of type \eqref{eq:nonemptyneighbours2} for vertices in $V^{(1)}$ we have the graph of Figure~\ref{fig:gen3}.
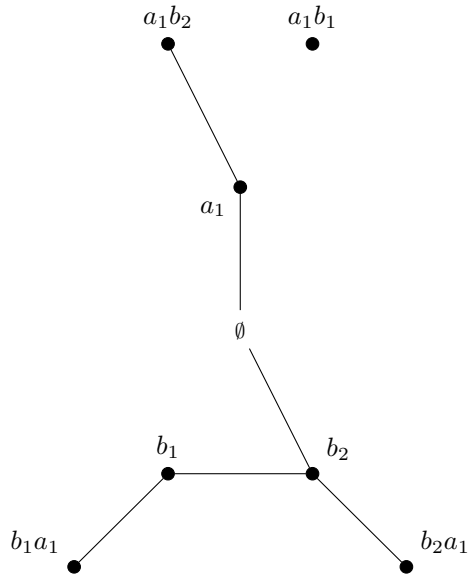
\begin{figure}[H]
\begin{center}
\begin{tikzpicture}[scale=0.95]
\node (0) at (-1,-1) [shape=circle,draw,scale=0.5,fill=black!100,label={[label distance=0cm]90:{\normalsize $b_1$}}] {};
\node (1) at (1,-1) [shape=circle,draw,scale=0.5,fill=black!100,label={[label distance=0cm]45:{\normalsize $b_2$}}] {};
\node (2) at (0,1) [shape=circle,scale=0.8] {$\emptyset$};
\node (3) at (0,3) [shape=circle,draw,scale=0.5,fill=black!100,label={[label distance=0cm]225:{\normalsize $a_1$}}] {};
\node (4) at (-1,5) [shape=circle,draw,scale=0.5,fill=black!100,label={[label distance=0cm]90:{\normalsize $a_1b_2$}}] {};
\node (5) at (1,5) [shape=circle,draw,scale=0.5,fill=black!100,label={[label distance=0cm]90:{\normalsize $a_1b_1$}}] {};
\node (6) at (2.3,-2.3) [shape=circle,draw,scale=0.5,fill=black!100,label={[label distance=0cm]45:{\normalsize $b_2a_1$}}] {};
\node (7) at (-2.3,-2.3) [shape=circle,draw,scale=0.5,fill=black!100,label={[label distance=0cm]135:{\normalsize $b_1a_1$}}] {};
\draw[-] (0) -- (1);
\draw[-] (1) -- (2);
\draw[-] (3) -- (4);
\draw[-] (2) -- (3);
\draw[-] (1) -- (6);
\draw[-] (0) -- (7);
\end{tikzpicture}
\end{center}
\caption{Another step in the genesis of $\Gamma_{1}\ast\Gamma_{2}$.}
\label{fig:gen3}
\end{figure}

\noindent
Now, $\textbf{u}(a_1b_1)\!=\!(a_1,\textcolor{blue}{b_1})$, $\textbf{u}(a_1b_2)\!=\!(a_1, \textcolor{blue}{b_2})$, $\textbf{u}(b_1a_1)\!=\!(\textcolor{blue}{a_1},b_1)$ and $\textbf{u}(b_2a_1)\!=\!(\textcolor{blue}{a_1},b_2)$.
\end{example}

\begin{example} Part \ref{item:cayley_graphs} of this example can be seen as motivating the definition of free products of graphs.
\begin{enumerate}[(i)]
\item Let $\Gamma_{i} = \pspicture(-0.1,-0.075)(1.1,0)\psdots(0,0)(1,0)\psline(0,0)(1,0)\rput(0.5,0.15){\scriptsize{$i$}}\endpspicture$, $i\in\{1,\dots,n\}$. Then $\Gamma=\Conv_{i=1}^{n}\Gamma_{i}$ is a tree with regular valency $n$ and the map of Remark \ref{rem:def_obs}\ref{item:sheet_colouring} defines a legal labelling of the tree in the sense of Burger--Mozes, see \cite[Section 3.2]{BM00a}. More generally, if $\Gamma_i$ is a regular tree of valency $d_i$, then $\Gamma$ is a regular tree of valency $\sum_{i=1}^{n}d_{i}$ and each vertex is adjacent to $d_i$ edges coloured~$i$.
\item\label{item:cayley_graphs} For finitely generated groups $(G_1,S_{1})$ and $(G_2,S_{2})$, the free product of the Cayley graphs of $G_{1}$ and $G_{2}$ is isomorphic to the Cayley graph of $G_{1}\!\ast\! G_{2}$:
\begin{displaymath}
  \displaybump \Gamma(G_1, S_1) \ast \Gamma(G_2, S_2) \cong \Gamma( G_1 \ast G_2, S_1 \sqcup S_2).
\end{displaymath}
Indeed, apply the definition by Pisanski--Tucker in Section \ref{sec:def_pt} to $\Gamma(G_{1},S_{1})$ and $\Gamma(G_{2},S_{2})$, rooted at the identity elements $e_{1}\in G_{1}$ and $e_{2}\in G_{2}$. In both $\Gamma(G_1, S_1) \ast \Gamma(G_2, S_2)$ and $\Gamma( G_1 \ast G_2, S_1 \sqcup S_2)$, the vertices are alternating words $w$ in $G_{1}\backslash\{e_{1}\}\sqcup G_{2}\backslash\{e_{2}\}$, and, considering multiplication with reduction in $G_{1}\ast G_{2}$, edges are precisely given by $\{w,ws\}$ for $s \in S_1 \sqcup S_2$.
\end{enumerate}
\end{example}

\begin{proposition}\label{prop:connected}
Let $\Gamma_{1},\ldots,\Gamma_{n}$ be connected graphs. Then $\Conv_{i=1}^{n}\Gamma_{i}$ is connected.
\end{proposition}

\begin{proof}
Let $(V,E):=\Conv_{i=1}^{n}\Gamma_{i}$. We show by induction on $l\ge 0$ that there is a path from $\emptyset\in V$ to $\tilde{v}=v_{1}\ldots v_{l}\in V$. If $l=0$ then $\tilde{v}=\emptyset$ and the statement holds. Now consider $\tilde{v}v_{l+1}\in V$. By the induction hypothesis, there is a path from $\emptyset$ to $v_{1}\ldots v_{l}$. Now, let $j\in\{1,\ldots,n\}$ be such that $v_{l+1}\in V(\Gamma_{j})$. Then both $\tilde{v}$ and $\tilde{v}v_{l+1}$ belong to the $\Gamma_{j}$-sheet $S_{(\tilde{v},j)}$. Since $\Gamma_{j}$ is connected there is a path from $\tilde{v}$ to $\tilde{v}v_{l+1}$ within that sheet. Concatenation yields the desired path from $\emptyset$ to $\tilde{v}v_{l+1}$.
\end{proof}

Note that the free product construction is commutative with unit a single vertex. Thus, \emph{we shall henceforth assume that no factor of a free product is a single vertex}. We show in Corollary~\ref{cor:associative} that it is also associative. However, while the graphs
\begin{displaymath}
 (\Gamma_{1}\ast\Gamma_{2})\ast\Gamma_{3},\ \Conv_{i=1}^{3}\Gamma_{i} \text{ and } \Gamma_{1}\ast(\Gamma_{2}\ast\Gamma_{3})
\end{displaymath}
are pairwise isomorphic, the sheet structure of a free product does depend on the bracketing, or whether no bracketing is specified: The graph $\Conv_{i=1}^{3}\Gamma_{i}$ has sheets of types $\Gamma_{1}$, $\Gamma_{2}$ and $\Gamma_{3}$ whereas $\Gamma_{1}\ast(\Gamma_{2}\ast\Gamma_{3})$ has sheets of type $\Gamma_{1}$ and $\Gamma_{2}\ast\Gamma_{3}$, and $(\Gamma_{1}\ast\Gamma_{2})\ast\Gamma_{3}$ has sheets of type $\Gamma_{1}\ast\Gamma_{2}$ and $\Gamma_{3}$. This is most relevant in Section~\ref{sec:aut_s}.

\subsection{M{\"o}ller-Seifter-Woess-Zemljic}\label{sec:def_mswz}

In \cite{MSWZ18}, M{\"o}ller {\it et al.\/} define a free product of connected, vertex-transitive graphs. Let $\Gamma_{1}=(V_{1},E_{1})$ and $\Gamma_{2}=(V_{2},E_{2})$ be connected, vertex-transitive graphs of order $m$ and $n$ respectively.

The free product is a connected sum of copies of $\Gamma_1$ and $\Gamma_2$. The copies are indexed by the vertices of the $(m,n)$-biregular tree $T_{m,n}$: bipartition $V(T_{m,n})=V_{m}\sqcup V_{n}$ into vertices of degree $m$ and $n$ respectively and assign a copy, $\smash{\Gamma_{1}^{(v)}}$, of $\Gamma_{1}$ for every $v\in V_{m}$ and a copy, $\smash{\Gamma_{2}^{(w)}}$, of $\Gamma_{2}$ for every $w\in V_{n}$. The sum is indexed by the edges of $T_{m,n}$: choosing bijections $\psi_{v}:V^{(v)}_{1}\to E_{T_{m,n}}(v)$, for every $v\in V_{m}$, and $\psi_{w}:V^{(w)}_{2}\to E_{T_{m,n}}(w)$, for every $w\in V_{n}$, associates a unique vertices in $V^{(v)}_{1}$ and $V^{(w)}_{2}$ to each edge $\{v,w\}\in E(T_{m,n})$. Then $\Gamma_{1}\ast\Gamma_{2}$ is the graph $\Gamma=(V,E)$
\begin{equation}
\label{eq:MSWZ}
 \underset{\{v,w\}\in E(T_{m,n})}{\Conn}\left(\Gamma_{1}^{(v)},\psi_{v}^{-1}(\{v,w\})\right)\#\left(\Gamma_{2}^{(w)},\psi_{w}^{-1}(\{v,w\})\right).
\end{equation}
The connected sum construction implies that $V(\Gamma)$ is in bijection with $E(T_{m,n})$ with $x\in V(\Gamma)$ corresponding to $\{v,w\}\in E(T_{m,n})$, which is formed by identifying $\psi_{v}^{-1}(\{v,w\})\in \Gamma_{1}^{(v)}$ with $\psi_{w}^{-1}(\{v,w\})\in \Gamma_{2}^{(w)}$. There is thus a $2$-$1$ map from the disjoint union of the graphs $\smash{\Gamma_{1}^{(v)}}$ and $\smash{\Gamma_{2}^{(w)}}$ to $V(\Gamma)$, and $V(\smash{\Gamma_{1}^{(v)}})$ and $V(\smash{\Gamma_{2}^{(w)}})$ will hence be regarded as being subsets of $V(\Gamma)$. The neighbours of $x$ are the edges $\psi_{v}(p)\in E_{T_{m,n}}(v)$, with $p\in V(\Gamma_{1}^{(v)})$ a neighbour of $\psi_{v}^{-1}(\{v,w\})$ and $\psi_{w}(q)\in E_{T_{m,n}}(w)$, with $q\in V(\Gamma_{2}^{(w)})$ a neighbour of $\psi_{w}^{-1}(\{v,w\})$.

In a similar spirit, Mohar \cite[Section 2]{Moh06} defines a free product of graphs with amalgamation, in which maps analogous to the maps $\psi$ above are used to amalgamate given subsets of $V(\Gamma_{1})$ and $V(\Gamma_{2})$.

\subsection{Pisanski-Tucker}\label{sec:def_pt}

In \cite{PT02}, Pisanski-Tucker define a free product of rooted graphs. Let $\Gamma_{1}=(V_{1},E_{1},v_{1})$ and $\Gamma_{2}=(V_{2},E_{2},v_{2})$ be rooted graphs. The free product $\Gamma_{1}\ast\Gamma_{2}$ is the graph $\Gamma=(V,E)$ with vertex set all finite words in the alphabet $V_{1}\backslash\{v_{1}\}\cup V_{2}\backslash\{v_{2}\}$ whose letters alternate between $V_{1}\backslash\{v_{1}\}$ and $V_{2}\backslash\{v_{2}\}$. Two vertices $w$ and $w'$ are adjacent if either $w'=w x$ for a neighbour $x$ of the respective root, or $w=ux$ and $w'=uy$ for adjacent $x,y$ in $\Gamma_1$ or $\Gamma_2$.

Thus multiple copies of $\Gamma_{1}$ and $\Gamma_{2}$ are connected by identifying each vertex in a copy of $\Gamma_{1}$ with the root in a copy of $\Gamma_{2}$, and each vertex in a copy of $\Gamma_{2}$ with the root in a copy of $\Gamma_{1}$.

\subsection{Quenell}\label{sec:def_q}

In \cite{Que94}, Quenell defines a free product of rooted graphs. Let $\Gamma_{1}=(V_{1},E_{1},v_{1}),\ldots,\Gamma_{n}=(V_{n},E_{n},v_{n})$ be rooted graphs. Recursively define graphs $B_{i}$ for $i\in\{1,\ldots,n\}$ by setting
\begin{displaymath}
 B_{i}=\underset{v\in V_{i}\backslash\{v_{i}\}}{\Conn}(\Gamma_{i},v)\#(B_{2},v_{2})\#\cdots\#\widehat{(B_{i},v_{i})}\#\cdots\#(B_{n},v_{n}),
\end{displaymath}
where $\widehat{(B_{i},v_{i})}$ means that $(B_{i},v_{i})$ is omitted from the connected sum. Finally, set $\Conv_{i=1}^{n}\Gamma_{i}:=(B_{1},v_{1})\#\cdots\#(B_{n},v_{n})$. 

When $n=2$, the graph $B_1$ in the Quenell construction corresponds to those words in the Pisanski-Tucker construction that are empty or begin with a letter in $V_{1}\backslash\{v_{1}\}$, and $B_2$ corresponds to those words that are empty or begin with a letter in $V_{2}\backslash\{v_{2}\}$. The Quenell construction connects $B_1$ and $B_2$ at their respective vertices corresponding to the empty word.

\section{Uniqueness and Quotients}\label{sec:uniqueness_quotients}

Based on Remark \ref{rem:def_obs} we define the following class of quotients of the free product $\Conv_{i=1}^{n}\Gamma_{i}$. It includes the free product itself as well as the cartesian product $\prod_{i=1}^{n}\Gamma_{i}$.

\begin{definition}\label{def:mashup}
Let $\Gamma_1$, \dots, $\Gamma_n$ be graphs. A \emph{mashup} of $\Gamma_1$, \dots, $\Gamma_n$ is a connected graph $M\!=\!(V,E)$ with embeddings $\varphi_{(x,i)}:\Gamma_{i}\to M$ ($x\!\in\! V,\ i\!\in\!\{1,\ldots,n\}$) such that
\begin{enumerate}[(a)]
\item\label{def:mashup_a}
$x\in \varphi_{(x,i)}(\Gamma_i)$, the \emph{$\Gamma_{i}$-sheet at $x$},\
\item\label{def:mashup_b}
$E=\bigcup\left\{ \varphi_{(x,i)}(E(\Gamma_i)) \mid i\in \{1,\dots,n\},\ x\in V\right\}$,
\item\label{def:mashup_c}
for all $i,j\in\{1,\dots, n\}$ and $x,y\in V$ the subgraphs $\varphi_{(x,i)}(\Gamma_i)$ and $\varphi_{(y,j)}(\Gamma_j)$ are either identical and $i=j$, have empty intersection, or intersect at a single vertex and $i\neq j$; and 
\item\label{def:mashup_d} for all $i,\, j\in\{1,\dots, n\}$ with $i\ne j$, $x\in V$ and $y\in V(\varphi_{(x,i)}(\Gamma_{i}))$ the maps
\begin{align*}
 \Gamma_i \to M\to \Gamma_i:&\ v \mapsto \varphi_{(x,i)}(v) \mapsto \varphi^{-1}_{(y,i)}(\varphi_{(x,i)}(v)) \\
 \displaybump \Gamma_i \to M\to \Gamma_j:&\ v \mapsto \varphi_{(x,i)}(v) \mapsto \varphi^{-1}_{(\varphi_{(x,i)}(v),j)}(\varphi_{(x,i)}(v))
\end{align*}
are constant.
\end{enumerate}
\end{definition}

\begin{remark}\label{rem:mashup}
Retain the notation of Definition~\ref{def:mashup}.
\begin{enumerate} 
\item\label{rem:mashup_1}
Conditions~\ref{def:mashup_a} and \ref{def:mashup_d} imply that every vertex of $M$ belongs to a unique $\Gamma_{i}$-sheet for every $i\in\{1,\ldots,n\}$.
\item\label{rem:mashup_2}
Conditions~\ref{def:mashup_b} and \ref{def:mashup_c} imply that every edge of $M$ belongs to a unique $\Gamma_{i}$-sheet for a unique $i\in\{1,\ldots,n\}$.
\item\label{rem:mashup_3}
Conditions~\ref{def:mashup_a} and \ref{def:mashup_d} provide a map $\mathbf{u}:V\to V(\Gamma_1)\times\dots\times V(\Gamma_n)$ defined by $\smash{\mathbf{u}(x)=(\varphi^{-1}_{(x,1)}(x),\dots, \varphi^{-1}_{(x,n)}(x))}$ such that all coordinates $u_{j}$ of $\mathbf{u}$ except the $i$-th one are constant on $\varphi_{(x,i)}(\Gamma_i)$. 
\item By conditions~\ref{def:mashup_b},\ref{def:mashup_c} and \ref{def:mashup_d}, every $x\in V$ has $|N_{\Gamma_i}(u_i(x))|$ neighbours in its $\Gamma_i$-sheet. The valency of~$x$ in $M$ is therefore given as in Remark\ref{rem:def_obs}\eqref{item:valency}.
\item\label{rem:mashup_4}
Conditions~\ref{def:mashup_a} and~\ref{def:mashup_d} imply~\ref{def:mashup_c}.
\end{enumerate}
\end{remark}

\begin{example}\label{ex:mashup}
Let $\Gamma_{1},\ldots,\Gamma_{n}$ be connected graphs. Examples of mashups are:
\begin{enumerate}[(i)]
 \item\label{item:mashup_free_product} The free product $\Gamma:=\Conv_{i=1}^{n}\Gamma_{i}$. For $x\in V(\Gamma)$ and $i\in\{1,\ldots,n\}$ define the map $\varphi_{(x,i)}:\Gamma_{i}\to\Gamma$ using the sheet structure of $\Gamma$ given in Remark~\ref{rem:def_obs}.
 \item\label{item:mashup_cartesian_product} The cartesian product $\Gamma:=\prod_{i=1}^{n}\Gamma_{i}$. Recall that $V(\Gamma)=\prod_{i=1}^{n}V(\Gamma_{i})$. For $x=(x_{1},\ldots,x_{n})\in V(\Gamma)$ and $i\in\{1,\ldots,n\}$, define the map $\varphi_{(x,i)}:\Gamma_{i}\to\Gamma$ by $\varphi_{(x,i)}(y):=(x_{1},\ldots,x_{i-1},y,x_{i+1},\ldots,x_{n})\in V(\Gamma)$ for every $y\in V(\Gamma_{i})$. In particular, $\mathbf{u}=\id$ in this case.
 \item\label{item:mashup_mixture} Parts~\ref{item:mashup_free_product} and~\ref{item:mashup_cartesian_product} can be combined via the following construction which resembles that of right-angled Artin groups, see {\it e.g.\/} \cite{Cha07}. Let $B$ be a non-empty, finite graph with vertex set $V(B)=\{1,\ldots,n\}$ and let $\Gamma_{i}=(V_{i},E_{i})$ be a connected graph for every $i\in V(B)$. We outline the construction of a mash-up $M_{B}=(V,E)$ of the $\Gamma_{i}$ $(i\in V(B))$ such that $M_{B}=\Conv_{i\in V(B)}\Gamma_{i}$ when $B$ has no edges, and $M_{B}=\prod_{i\in V(B)}\Gamma_{i}$ when $B$ is complete. The vertex set of $M_{B}$ consists of certain admissible words in the alphabet $V_{1}\sqcup\cdots\sqcup V_{n}$, forming a strict subset of those defined in Section~\ref{sec:def_new}.
 
 \vspace{0.2cm}
 \noindent
 \emph{Vertices}. The set of admissible words of length $l$ is denoted by $V^{(l)}$ and defined by induction on $l$. It involves an update function $\mathbf{u}\!:V\!\to V_{1}\times\cdots\times V_{n}$ defined inductively alongside $V^{(l)}$. Choose $(u_{1},\ldots,u_{n})\in V_{1}\times\cdots\times V_{n}$. Set
 \begin{displaymath}
  \displaybump V^{(0)}:=\{\emptyset\} \text{ and } \mathbf{u}(\emptyset):=(u_{1},\ldots,u_{n})\in V_{1}\times\cdots\times V_{n}.
 \end{displaymath}
 Now, assume that $V^{(k)}$ and $\mathbf{u}|_{V^{(k)}}$ are defined for $k\in\{0,\ldots,l\}$. Given a vertex $\tilde{v}=v_{1}\ldots v_{k}\in V^{(k)}$, let $\mathbf{u}(\tilde{v})=(u_{1}(\tilde{v}),\ldots,u_{n}(\tilde{v}))$. The set $V^{(l+1)}$ consists of all words of the form $\tilde{v}v_{l+1}$ where $\tilde{v}=v_{1}\ldots v_{l}\in V^{(l)}$ and, supposing that $v_{l}\in V_{i}$ and $v_{l+1}\in V_{j}$ we have
 \begin{displaymath}
  \displaybump j\neq i,\ v_{l+1}\neq u_{j}(\tilde{v}) \text{ and } j>i \text{ whenever } \{i,j\}\in E(B).
 \end{displaymath}
 The update function is defined on $V^{(l+1)}$ by 
 \begin{displaymath}
  \displaybump u_{h}(\tilde{v}v_{l+1}):=\begin{cases}u_{h}(\tilde{v}) & \text{if } h\neq j \\ v_{l+1} & \text{if } h=j\end{cases}.
 \end{displaymath}
 
 \vspace{0.2cm}
 \noindent
 \emph{Sheets}. To define the edges of $M_{B}$ and show that it is a mash-up of the $\Gamma_{i}$ $(i\in V(B))$, we define its sheets: For $\tilde{v}=v_{1}\ldots v_{l}\in V$ with $v_{l}\in V_{i}$ and $j\in\{1,\ldots,n\}$, the $\Gamma_{j}$-sheet at $\tilde{v}$ has the following vertex set $S\subseteq V$, which is naturally turned into a copy of $\Gamma_{j}$.
 \begin{enumerate}[$\bullet$,leftmargin=0.33cm]
  \item If $j=i$, let $S:=\{v_1\dots v_{l-1}\}\cup\left\{v_1\dots v_{l-1}v'_{l}\mid  v'_l\in V_j\setminus\{u_j(v_1\dots v_{l-1})\}\right\}$.
  \item If $j>i$ and $\{i,j\}\in E(B)$, or $\{i,j\}\not\in E(B)$, set
  \begin{displaymath}
   \displaybump S:=\{\tilde{v}\}\cup\{\tilde{v}v_{l+1}\mid v_{l+1}\in V_{j}\backslash\{u_{j}(\tilde{v})\}\}.
  \end{displaymath}
  \item If $j<i$ and $\{i,j\}\in E(B)$, define $\iota(v_{m})\in\{1,\ldots,n\}$ $(m\in\{1,\ldots,l\})$ as the unique index such that $v_{m}\in\Gamma_{\iota(v_{m})}$. Now, let $k\in\{1,\ldots,l\}$ be minimal such that $\iota(v_{k})\ge j$ and $\{\iota(v_{r}),j\}\in E(B)$ for all $r\!\in\!\{k,\ldots,l-1\}$. If $\iota(v_{k})=j$, then $k\neq l$ and we set
  \begin{displaymath}
   \displaybump S:=\{\tilde{v}\}\cup\{v_{1}\ldots v_{k-1}v'v_{k+1}\ldots v_{l}\mid v'\in V_{j}\backslash\{u_{j}(\tilde{v})=v_{k}\}\}.
  \end{displaymath}
  Otherwise, if $\iota(v_{k})>j$, then $k\le l$ and we set
  \begin{displaymath}
   \displaybump S:=\{\tilde{v}\}\cup\{v_{1}\ldots v_{k-1}v'v_{k}\ldots v_{l}\mid v'\in V_{j}\backslash\{u_{j}(\tilde{v})\}\}.
  \end{displaymath}
 \end{enumerate}
 In each of the above cases, the set $S$ is in natural bijection with $V_{j}$. The edges of the $\Gamma_{j}$-sheet at $\tilde{v}$ are then defined using $E_{j}$ via said bijection.
 
 \vspace{0.2cm}
 \noindent
 For example, one can see that the Cayley graph of the right-angled Artin group $A:=\{s_{1},\ldots,s_{n}\mid \{[s_{i},s_{j}]\mid \{i,j\}\in E(B)\}\}$ defined via the base graph $B$ is isomorphic to $M_{B}$ where $\Gamma_{i}=T_{2}$ for all $i\in\{1,\ldots,n\}$.
\end{enumerate}
\end{example}

Now, let $M$ be a mashup of $\Gamma_1,\ldots,\Gamma_n$ and let $\gamma=(e_{1},\ldots,e_{k})$ be a path in $M$. By Remark~\ref{rem:mashup}\eqref{rem:mashup_2}, there is a unique index $c(e_{m})\in\{1,\ldots,n\}$ for every edge~$e_m$ in $\gamma$ such that~$e_m$ belongs to a $\Gamma_{c(e_{m})}$-sheet of $M$. Write $e_{m}=\{x_{m-1},x_{m}\}\in E(\Gamma_{c(e_{m})})$. We say that $\gamma$ \emph{transitions from $\Gamma_i$ to $\Gamma_j$ at~$m$} if $c(e_m)=i$, $c(e_{m+1})=j$ and $i\ne j$. Then $x_m$ is a \emph{transition point in~$M$} and $\smash{\varphi_{(x_m,i)}^{-1}(x_m)}$ is a \emph{transition point in $\Gamma_i$}.

In the case $\smash{M=\Conv_{i=1}^n \Gamma_i}$, there is a path from the empty word $\emptyset$ to $\tilde{v}=v_1\dots v_l$ with transitition points $v_k$ ($k\in\{1,\dots,l-1\}$). If $v_k\in V(\Gamma_{i_k})$, then the path goes from $v_1\dots v_{k-1}$ to $v_1\dots v_k$ in a $\Gamma_{i_k}$-sheet and transitions into a $\Gamma_{i_{k+1}}$-sheet at $v_{1}\ldots v_k$. Denote $v_1\dots v_k$ by $x_k$. The length of the shortest path from $\emptyset$ to $\tilde{v}$ now is
$$
d(\emptyset, \tilde{v}) = \sum_{k=1}^l d_{\Gamma_{i_k}}(\varphi^{-1}_{(x_{k},i_k)}(x_{k-1}),\varphi^{-1}_{(x_k,i_k)}(x_k)) = \sum_{k=1}^l d_{\Gamma_{i_k}}(u_{i_k}(x_{k-1}),u_{i_k}(x_{k})),
$$
where, it may be recalled, $u_{i_k}(x_{k}) = v_k$. Moreover, we have the following.

\begin{lemma}\label{lem:paths}
Let $\Gamma_{1},\ldots,\Gamma_{n}$ be connected graphs, $\Gamma:=\Conv_{i=1}^{n}\Gamma_{i}$, and $\tilde{v},\tilde{w}\in V(\Gamma)$. Then all paths of minimal length from $\tilde{v}$ to $\tilde{w}$ have the same transition points.
\end{lemma}

\begin{proof}
If $\gamma$ transitions at~$m$ then $\tilde{x}_{m}$ is a cut vertex and deleting $\tilde{x}_m$ disconnects $\Gamma$ so that $\tilde{v}$ lies in one component and $\tilde{w}$ in the other. Indeed, if $\tilde{w}$ lay in the same component as~$\tilde{v}$, then $\gamma$ would return to that component through $\tilde{x}_m$ and could be shortened. Every path from $\tilde{v}$ to~$\tilde{w}$ must therefore pass through~$\tilde{x}_m$.
\end{proof}

The following proposition states in particular that every mashup is a quotient of its associated free product.

\begin{proposition}
\label{prop:mashup_quotientmap}
Let $M$ be a mashup of $\Gamma_1$, \dots, $\Gamma_n$. Further, let $x_0\in V(M)$ and $\smash{u_i:=\varphi^{-1}_{(x_0,i)}(x_0)}$ ($i\in\{1,\dots, n\}$). Initialise $\Gamma:=\smash{\Conv_{i=1}^n \Gamma_i}$ by $\mathbf{u}(\emptyset) = (u_1,\dots, u_n)$. Then there is a surjective graph homomorphism $\smash{\psi:\Gamma\to M}$ with $\psi(\emptyset) = x_0$ and $\mathbf{u}(\psi(\tilde{v})) = \mathbf{u}(\tilde{v})$ for all $\tilde{v}\in V(\Gamma)$.
\end{proposition}

\begin{proof}
We argue by induction on the natural number $l$ in $V(\Gamma)=\bigcup_{l\in\bbN_{0}}V^{(l)}$. Define $\psi(\emptyset) = x_0$ and extend $\psi$ to $V^{(1)}$ by setting 
\begin{equation}\label{eq:quotientmap}
\psi(v_1) := \varphi_{(x_0,i)}(v_1) \mbox{ if }v_1\in \Gamma_i\setminus\{u_i = \varphi^{-1}_{(x_0,i)}(x_0)\}.
\end{equation}
Then $\psi$ is a morphism from $\mathrm{span}_{\Gamma}(V^{(0)}\cup V^{(1)})$ to $M$ since $\psi(\emptyset)=x_0=\varphi_{(x_0,i)}(u_i)$, by Definition~\ref{sec:def_new}~\eqref{eq:emptyneighbours} of the neighbours of $\emptyset\in V(\Gamma)$, and since $\varphi_{(x_0,i)}$ is a morphism.

By Definition~\ref{def:mashup}\ref{def:mashup_d}, the map $\smash{y\mapsto \varphi^{-1}_{(y,j)}(y)}$ is constant on $\psi(S_{(\emptyset,i)}) = \varphi_{(x_0,i)}(\Gamma_i)$ for any $j\in\{1,\ldots,n\}\backslash\{i\}$. The definition of the initial value of $\mathbf{u}$ implies that this constant is equal to $\smash{\varphi^{-1}_{(x_0,j)}(x_0) = u_j = u_j(\emptyset)}$.

Now, suppose $\psi$ is defined on $\mathrm{span}_{\Gamma}(\bigcup_{k=0}^{l}V^{(k)})$ and for $y = \psi(v_1\dots v_k)$ satisfies
\begin{equation}\label{eq:recursion}
\varphi^{-1}_{(y,i)}(y) = \begin{cases}
v_k & \mbox{ if }v_k\in V(\Gamma_{i})\\
u_i(v_1\dots v_k) & \mbox{ if }v_k\in V(\Gamma_{j}),\ j\ne i
\end{cases},\qquad (k\in\{1,\dots, l\}). 
\end{equation} 
This holds for $l=1$ by the above. Now, consider $\tilde{v}v_{l+1}=v_1\dots v_lv_{l+1}\in V^{(l+1)}$, where $v_l\in V_i$, $v_{l+1}\in V_j$ and $i\ne j$. In particular, $v_{l+1}\ne u_j(\tilde{v})$. Set $y:=\psi(\tilde{v})$ and
$$
\psi(\tilde{v}v_{l+1}):=\varphi_{(y,j)}(v_{l+1}).
$$
Since~\eqref{eq:recursion} implies that $\varphi_{(y,j)}(u_j(\tilde{v})) = y$, it follows that $\psi$ is a morphism from $S_{(\tilde{v},j)}$ to the $\Gamma_j$-sheet of $M$ containing $y$. By~\eqref{eq:nonemptyneighbours2} of $E(\Gamma)$, this extension of $\psi$ is a morphism $\mathrm{span}_{\Gamma}(\smash{\bigcup_{k=0}^{l+1} V^{(k)})\to M}$. To see that~\eqref{eq:recursion} extends to the case $k = l+1$, observe that $\smash{\varphi^{-1}_{(z,j)}(z) = v_{l+1}}$ for $z = \psi(\tilde{v}v_{l+1})$ by definition and, for $h\ne j$, 
$$
\varphi^{-1}_{(z,h)}(z) = \varphi^{-1}_{(y,h)}(y) = \begin{cases}
v_l = u_i(\tilde{v}) & \mbox{ if }h=i\\
u_h(\tilde{v}) & \mbox{ if }h\ne i
\end{cases},
$$
where the first equality holds by Definition~\ref{def:mashup}\ref{def:mashup_d} and the second row of the second equality holds by~\eqref{eq:recursion} with $k=l$. Since $u_h(\tilde{v}v_{l+1}) = u_h(\tilde{v})$ for $h\ne j$,~\eqref{eq:recursion} extends.

Continuing recursively produces the claimed morphism $\psi$, which is surjective by Lemma~\ref{lem:quotientmap}, the connectedness of $M$, and conditions~\ref{def:mashup_a} and~\ref{def:mashup_b} in Definition~\ref{def:mashup}.
\end{proof}

Using Example~\ref{ex:mashup}\ref{item:mashup_cartesian_product}, the following is a consequence of Proposition \ref{prop:mashup_quotientmap}.

\begin{corollary}\label{cor:quotientmap}
Let $\Gamma_{1},\ldots,\Gamma_{n}$ be connected graphs. Then $\mathbf{u}:\!\Conv_{i=1}^n \Gamma_i \!\to\! \Gamma_1\times \dots \times \Gamma_n$ is a graph homomorphism. \qed
\end{corollary}

We now show that the isomorphism class of a free product is independent of the initialisation of its update function.

\begin{corollary}\label{cor:uniqueness}
Let $\Gamma_{1},\ldots,\Gamma_{n}$ be connected graphs and let $\Gamma$ and $\Gamma'$ be the free products of $\Gamma_{1},\ldots,\Gamma_{n}$ with respect to $(u_1,\dots, u_n)$ and $(u_1',\dots, u_n')$ in $\prod_{i=1}^{n}V(\Gamma_{i})$. Then $\Gamma$ and $\Gamma'$ are isomorphic. 
\end{corollary}

\begin{proof}
Let $\mathbf{u}$ and $\mathbf{u}'$ denote the update functions on $\Gamma$ and $\Gamma'$ respectively. Choose $\tilde{v}_{0}\in V(\Gamma')$ such that $\mathbf{u}'(\tilde{v}_{0}) = (u_1,\dots, u_n)$. Since $\Gamma'$ is a mashup of $\Gamma_1$, \dots, $\Gamma_n$, there is, by Proposition~\ref{prop:mashup_quotientmap}, a surjective homomorphism $\psi:\Gamma\to \Gamma'$ with $\psi(\emptyset) = \tilde{v}_{0}$. 

To see that $\psi$ is injective, we construct a left inverse: Interchanging the roles of $\Gamma$ and $\Gamma'$, Proposition~\ref{prop:mashup_quotientmap} also yields a surjective homomorphism $\psi':\Gamma'\to\Gamma$. Since $\psi'$ is surjective, there is $\tilde{v}'_0 \in V(\Gamma')$ such that $\psi'(\tilde{v}'_0) = \emptyset\in V(\Gamma)$. Proposition~\ref{prop:mashup_quotientmap} ensures that $\mathbf{u}'(\tilde{v}'_0) = \mathbf{u}(\psi'(\tilde{v}_{0}')) = \mathbf{u}(\emptyset) = (u_1,\dots, u_n)$. Hence $\tilde{v}_0$ from the previous paragraph may be chosen to be $\tilde{v}'_0$. Then the composition $\psi'\circ \psi : \Gamma\to \Gamma$ is surjective and satisfies $\psi'\circ\psi(\emptyset) = \emptyset$ as well as $\mathbf{u}(\psi'\circ\psi(\tilde{v})) = \mathbf{u}(\tilde{v})$ for every $\tilde{v}\in V(\Gamma)$. 

It remains to show that $\psi'\circ\psi$ is the identity homomorphism. We argue by induction on the distance of a vertex from the empty word, given that $\psi'\circ\psi(\emptyset) = \emptyset$. Suppose $\psi'\circ\psi(\tilde{v}) =\tilde{v}$ for every $\tilde{v}$ with $d(\emptyset,\tilde{v})\leq d$. Consider $\tilde{w}$ with $d(\emptyset,\tilde{w})= d+1$. Then there is a neighbour, $\tilde{v}$ of $\tilde{w}$ with $d(\emptyset,\tilde{v}) = d$ and a unique $i\in\{1,\dots,n\}$ such that $\{\tilde{v},\tilde{w}\}$ belongs to $S_{(\tilde{v},i)}$. Since $\psi'\circ\psi(\tilde{v}) = \tilde{v}$ and $u_i(\psi'\circ\psi(\tilde{w})) = u_i(\tilde{w})$, it follows that $\psi'\circ\psi(\tilde{w})$ equals $\tilde{w}$.
\end{proof}

The proof of Corollary~\ref{cor:uniqueness} entails that if $\tilde{v}\in V(\Conv_{i=1}^n \Gamma_i)$ satisfies $\mathbf{u}(\tilde{v}) = \mathbf{u}(\emptyset)$, then there is an automorphism $\alpha\in\Aut(\Conv_{i=1}^n \Gamma_i)$ such that $\alpha(\emptyset) = \tilde{v}$ and 
\begin{equation}\label{eq:intertwine}
\mathbf{u}(\alpha(\tilde{w})) = \mathbf{u}(\tilde{w})\mbox{ for every }\tilde{w}\in V(\Conv_{i=1}^n \Gamma_i).
\end{equation}

\begin{corollary}\label{cor:associative}
Let $\Gamma_{1},\Gamma_{2},\Gamma_{3}$ be connected graphs and $\Gamma:=\Conv_{i=1}^{3}\Gamma_{i}$. Then
\begin{displaymath}
 (\Gamma_{1}\ast\Gamma_{2})\ast\Gamma_{3}\cong\Gamma\cong\Gamma_{1}\ast(\Gamma_{2}\ast\Gamma_{3}).
\end{displaymath}
\end{corollary}

\begin{proof}
The graph $\Gamma':=(\Gamma_{1}\ast\Gamma_{2})\ast\Gamma_{3}$ is naturally a mashup of the graphs $\Gamma_{1},\Gamma_{2},\Gamma_{3}$, where the embeddings $\varphi_{(x,i)}:\Gamma_{i}\to\Gamma'$ ($x\in V(\Gamma')$, $i\in\{1,2,3\}$) come from the structure of $\Gamma'$ as a free product of $\Gamma_{1}\ast\Gamma_{2}$ and $\Gamma_{3}$, and the structure of $\Gamma_{1}\ast\Gamma_{2}$ as a free product of $\Gamma_{1}$ and $\Gamma_{2}$. Hence, by Proposition~\ref{prop:mashup_quotientmap}, there is a surjective graph homomorphism $\psi:\Gamma\to\Gamma'$ intertwining the update functions. As a consequence, $\psi$ is completely determined, and in particular injective, on balls of arbitrary radius around $\emptyset\in V(\Gamma)$ by induction. Hence it is injective overall. Proceed analogously in the case of $\Gamma_{1}\ast(\Gamma_{2}\ast\Gamma_{3})$.
\end{proof}

\section{Comparison of definitions of the free product}\label{sec:equivalence}
Here, we show that the definitions presented in Sections~\ref{sec:def_new} to \ref{sec:def_q}, which we will denote, respectively, by $\ast_{1}$ up to $\ast_{4}$ in this section, are equivalent when the graphs are vertex-transitive. First, we highlight the obvious differences, namely: $\ast_{2}$ is defined only for pairs of graphs that are vertex-transitive; and $\ast_{3}$ and $\ast_{4}$ are defined only for rooted graphs and produce rooted graphs. The equivalence of $\ast_{3}$ and $\ast_{4}$ with $\ast_{1}$ arises because all choices of root vertex are equivalent when a graph is vertex-transitive. 

We now turn to the various constructions. In all cases, vertex-transitivity of the factors is required to account for the fact that in the genesis of all graphs, new sheets are attached to the existing part at different vertices.

\begin{proposition}
\label{prop:1&2}
Let $\Gamma_{1}$ and $\Gamma_{2}$ be connected, vertex-transitive graphs. Then the graphs $\Gamma_{1}\ast_{1}\Gamma_{2}$ and $\Gamma_{1}\ast_{2}\Gamma_{2}$ are isomorphic.
\end{proposition}

\begin{proof}
Let $\Gamma':=\Gamma_{1}\ast_{2}\Gamma_{2}$. Recall that there is a bijection $V(\Gamma')\to E(T_{m,n})$ with the vertex $x\in V(\Gamma')$, corresponding to $\{v,w\}\in E(T_{m,n})$, being defined by identifying $\psi^{-1}_{v}(\{v,w\})\in V(\Gamma^{(v)})$ with $\psi^{-1}_{w}(\{v,w\})\in V(\Gamma^{(w)})$. Pick $x_{0}\in V(\Gamma')$ corresponding to $\{v_0,w_0\}\in E(T_{m,n})$ and initialise $\Gamma:=\Gamma_{1}\ast_{1}\Gamma_{2}$ with 
$$
\smash{\mathbf{u}(\emptyset)=(\psi^{-1}_{v_0}(\{v_0,w_0\}),\psi^{-1}_{w_0}(\{v_0,w_0\}))}.
$$  

Recall that $V^{(l)}$ denotes the set of admissible words of length~$l$ in $V(\Gamma)$ and define $\varphi:\Gamma\to\Gamma'$ inductively on $l$ as follows. Set $\varphi(\emptyset)=x_{0}$. Next, $\smash{p\in V^{(1)}}$ lies on either the $\Gamma_1$-sheet or the $\Gamma_2$-sheet at $\emptyset$, that is, either $p\in V(\Gamma_1)\setminus\{ \psi^{-1}_{v_0}(\{v_0,w_0\})\}$ or $p\in V(\Gamma_2)\setminus\{\psi^{-1}_{w_0}(\{v_0,w_0\})\}$. Define $\varphi(p) = p\in V(\Gamma^{(v_0)})\subset V(\Gamma')$ in the first case and $\varphi(p) = p\in V(\Gamma^{(w_0)})\subset V(\Gamma')$ in the second. Then $\varphi$ is an isomorphism
$$
\mathrm{span}_{\Gamma}(V^{(0)}\cup V^{(1)})\to\bigl(\Gamma^{(v_0)},\psi^{-1}_{v_0}(\{v_0,w_0\})\bigr)\# \bigl(\Gamma^{(w_0)},\psi^{-1}_{w_0}(\{v_0,w_0\})\bigr)
$$ 
Now suppose that $\varphi$ has been defined on words up to length $l\in\bbN_{\ge 1}$ and consider $\tilde{p} = p_{1}\ldots p_{l} \in V^{(l)}$. Without loss of generality, assume that $p_{l}\in V(\Gamma_{1})$. Then the $\Gamma_2$-sheet of $\Gamma$ at $\tilde{p}$ is attached at $u_{2}(\tilde{p})$, while, if $\varphi(\tilde{p})\in V(\Gamma')$ corresponds to $\{v,w\}\in E(T_{m,n})$, then $\varphi(\tilde{p})$ is identified with $\psi^{-1}_{w}(\{v,w\})\in V(\Gamma^{(w)})$. Since $\Gamma_{2}$ is vertex-transitive, there is $\alpha_{\tilde{p}}\in\Aut(\Gamma_{2})$ such that $\smash{\alpha_{\tilde{p}}(u_{2}(\tilde{p}))}=\psi^{-1}_{w}(\{v,w\})$. Choose such $\alpha_{\tilde{p}}$ for each ${\tilde{p}}\in V^{(l)}$ and extend the definition of $\varphi$ to ${\tilde{p}}p_{l+1}\in V^{(l+1)}$ by 
$$
\varphi(\tilde{p}p_{l+1}):=\begin{cases}
\alpha_{\tilde{p}}(p_{l+1})\in V(\Gamma^{(w)})\setminus \{\psi^{-1}_{w}(\{v,w\})\}, & \mbox{ if }p_l\in V(\Gamma_{1})\\
\alpha_{\tilde{p}}(p_{l+1})\in V(\Gamma^{(v)})\setminus \{\psi^{-1}_{v}(\{v,w\})\}, & \mbox{ if }p_l\in V(\Gamma_{2})
\end{cases}.
$$ 
Since $\alpha_{\tilde{p}}$ is an automorphism for each ${\tilde{p}}\in V^{(l)}$, it follows that $\varphi$ is a graph isomorphism.
\end{proof}

As an immediate consequence of Corollary~\ref{cor:uniqueness} and Proposition~\ref{prop:1&2} we see that the definition of $\ast_{2}$ does not depend on the choice of the bijections involved.
\begin{corollary}
\label{cor:2_well_defined}
Let $\Gamma_{1}$ and $\Gamma_{2}$ be connected, vertex-transitive graphs. Then the isomorphism type of $\Gamma_{1}\ast_{2}\Gamma_{2}$ is well-defined.
\end{corollary}

As to the definition of Pisanski--Tucker in Section~\ref{sec:def_pt}, we have the following.

\begin{proposition}
Let $(\Gamma_{1},r_{1})$ and $(\Gamma_{2},r_{2})$ be connected, vertex-transitive, rooted graphs. Then the graphs $\Gamma_{1}\ast_{1}\Gamma_{2}$ and $\Gamma_{1}\ast_{3}\Gamma_{2}$ are isomorphic.
\end{proposition}

\begin{proof}
Let $\Gamma':=\Gamma_{1}\ast_{3}\Gamma_{2}$. Initialise $\Gamma:=\Gamma_{1}\ast_{1}\Gamma_{2}$ with $u(\emptyset)=(r_{1},r_{2})$. Now define $\varphi:\Gamma\to\Gamma'$ as follows: Set $\varphi(\emptyset)=\varepsilon$, where $\varepsilon\in V(\Gamma')$ is the empty string, and assume that $\varphi$ has been defined on words up to length $l\in\bbN_{0}$. Consider $\tilde{v}=v_{1}\ldots v_{l}\in V(\Gamma)$ and suppose $\tilde{v}v_{l+1}\in V(\Gamma)$, where $v_{l+1}\in V(\Gamma_{i})$. Since $\Gamma_{i}$ is vertex-transitive, there is $\alpha\in\Aut(\Gamma_{i})$ such that $\alpha(u_{i}(\tilde{v}))=r_{i}$. Define $\varphi(\tilde{v}v_{l+1})=\varphi(\tilde{v})\alpha(v_{l+1})$. Since all the maps $\alpha$ are automorphisms, it follows that $\varphi$ is a graph isomorphism.
\end{proof}

A similar argument applies in the case of Quenell's definition from Section~\ref{sec:def_q}.

\begin{proposition}
Let $(\Gamma_{1},r_{1})$ and $(\Gamma_{2},r_{2})$ be connected, vertex-transitive, rooted graphs. Then the graphs $\Gamma_{1}\ast_{1}\Gamma_{2}$ and $\Gamma_{1}\ast_{4}\Gamma_{2}$ are isomorphic. \qed
\end{proposition}

\section{Automorphisms of Free Product Graphs}

It was seen in Equation~\ref{eq:intertwine} that free product graphs admit many automorphisms which preserve the update function. Three additional types of automorphisms of free product graphs are described in this section.

The first proposition shows that further automorphisms arise from automorphisms of the free product's factors.
\begin{proposition}\label{prop:aut_intertwine}
Let $\Gamma_{1},\ldots,\Gamma_{n}$ be connected graphs and initialise $\Gamma:=\Conv_{i=1}^n \Gamma_i$ with $\mathbf{u}(\emptyset)=(u_1,\dots, u_n)\in\prod_{i=1}^{n}V(\Gamma_{i})$. Given $j\!\in\!\{1,\ldots,n\}$ and $\alpha\!\in\! \Aut(\Gamma_j)$ define
\begin{displaymath}
 \tilde{\alpha}:\bigsqcup_{i=1}^{n}V(\Gamma_{i})\to\bigsqcup_{i=1}^{n}V(\Gamma_{i}),\ v\mapsto\begin{cases}v & \text{if }v\in V(\Gamma_{i}),\ i\neq j \\ \alpha(v) & \text{if }v\in V(\Gamma_{j})\end{cases}.
\end{displaymath}
Let $\tilde{v}\!=\!v_{1}\ldots v_{l}\in V(\Gamma)$. When $\alpha(u_j)\!=\!u_j$, let $\hat{\alpha}(\tilde{v})\!:=\!\tilde{\alpha}(v_1)\dots \tilde{\alpha}(v_l)$. Otherwise, set
\begin{displaymath}
\hat{\alpha}(\tilde{v}) := \begin{cases} 
\alpha(u_j)\tilde{\alpha}(v_1)\dots \tilde{\alpha}(v_l) &  \text{if } v_1\not\in \Gamma_j \mbox{ or } l=0\\
\tilde{\alpha}(v_2)\dots \tilde{\alpha}(v_l) & \text{if } v_1\in \Gamma_j\mbox{ and }\alpha(v_1) = u_j\\
\tilde{\alpha}(v_1)\dots \tilde{\alpha}(v_l) & \text{otherwise}
\end{cases}.
\end{displaymath}
Then $\hat{\alpha}\!\in\!\Aut(\Gamma)$ and the map $\Aut(\Gamma_{j})\!\to\!\Aut(\Gamma),\ \alpha\mapsto\hat{\alpha}$ is a group monomorphism. Moreover, $\mathbf{u}\circ\hat{\alpha} = \overline{\alpha}\circ\mathbf{u}$, where $\overline{\alpha} = \id\times\dots\times\alpha\times\dots\times\id\in \Aut(\prod_{i=1}^n \Gamma_i)$.  
\end{proposition} 

\begin{proof}
We first show that $\hat{\alpha}(\tilde{v})\!\in\! V(\Gamma)$. Since $\tilde{\alpha}$ preserves $V(\Gamma_{i})$ for all $i\!\in\!\{1,\ldots,n\}$, any two consecutive letters in the word $\hat{\alpha}(\tilde{v})$ belong to distinct graphs. As to the other half of Equation~\eqref{eq:admissible}, note that only the $j$-th component of $\mathbf{u}$ is affected by~$\hat{\alpha}$. Let $k_1<\cdots<k_r\in\{1,\ldots,l\}$ be the indices for which $v_{k_s}\in V(\Gamma_{j})$ ($s\in\{1,\ldots,r\}$).

First, consider the case when $\alpha(u_j) = u_j$. Since $\tilde{v}\in V(\Gamma)$, we have $v_{k_1}\ne u_j$ and $v_{k_{s+1}}\ne v_{k_s}$ for $s\geq 1$. Then $\alpha(v_{k_1})\ne u_j$ and $\alpha(v_{k_{s+1}})\ne \alpha(v_{k_s})$ for $s\geq1$ because $\alpha(u_{j})=u_{j}$ and $\alpha$ is a bijection. Hence $\hat{\alpha}(\tilde{v})\in V(\Gamma)$. The assertion $\mathbf{u}\circ\hat{\alpha} = \overline{\alpha}\circ\mathbf{u}$ is immediate from the definition of $\hat{\alpha}$.

Now suppose $\alpha(u_j) \ne u_j$. There are several cases to consider: If $l=0$ or $v_1\not\in \Gamma_j$, then the first letter of $\hat{\alpha}(\tilde{v})$ is $\alpha(u_j)$, which belongs to $\Gamma_j$ and is not equal to $u_j$. The remaining letters from $\Gamma_j$ are $\alpha(v_{k_s})$ ($s\geq1$) and satisfy $\alpha(v_{k_1})\ne \alpha(u_j)$ and $\alpha(v_{k_{s+1}}) \ne \alpha(v_{k_s})$ ($s\geq1$) because $\alpha$ is a bijection. If $v_1\in \Gamma_j$ and $\alpha(v_1) = u_j$, then the letters from $\Gamma_j$ are $\alpha(v_{k_s})$ ($s\geq2$) and satisfy $\alpha(v_{k_2})\ne u_j$ and $\alpha(v_{k_{s+1}}) \ne \alpha(v_{k_s})$ ($s\geq2$) because $v_{k_{2}}\neq v_{1}$ and $\alpha$ is a bijection. Finally, if $v_1\in \Gamma_j$ and $\alpha(v_1)\ne u_j$, then $k_1 = 1$ and we have that $\alpha(v_{k_1})\ne u_j$ and $\alpha(v_{k_{s+1}})\ne \alpha(v_{k_s})$ ($s\geq1$). Again, we conclude $\mathbf{u}\circ\hat{\alpha} = \overline{\alpha}\circ\mathbf{u}$.

To see that $\hat{\alpha}\in\Aut(\Gamma)$, recall that every edge of $\Gamma$ belongs to a unique sheet. Since $\hat{\alpha}$ preserves $\Gamma_j$-sheets and acts like $\alpha$ on every such sheet, edge relations within $\Gamma_j$-sheets are preserved. Furthermore, $\hat{\alpha}$ permutes the $\Gamma_i$-sheets attached to a given $\Gamma_j$-sheet and acts like the identity on every such $\Gamma_{i}$-sheet. Hence edge-relations in $\Gamma_{i}$-sheets ($i\neq j$) are preserved as well. We conclude that $\hat{\alpha}\in\Aut(\Gamma)$.

The fact that $\hat{\alpha}$ preserves $\Gamma_j$-sheets and acts like $\alpha$ on every such sheet implies that the map $\Aut(\Gamma_{i})\to\Aut(\Gamma)$, $\alpha\mapsto \hat{\alpha}$ is an injective homomorphism.
\end{proof}

\begin{remark}
Retain the notation of Proposition~\ref{prop:aut_intertwine}. In the context of cartesian products, images of distinct embeddings $\Aut(\Gamma_{j})\!\to\!\Aut(\prod_{i=1}^{n}\Gamma_{i})$, $\alpha\mapsto\overline{\alpha}$ commute, whereas commutators of automorphisms in distinct images under the embeddings $\Aut(\Gamma_{j})\to\Aut(\Conv_{i=1}^{n}\Gamma_{i})$, $\alpha\mapsto\hat{\alpha}$ yield automorphisms of the type in Equation~\eqref{eq:intertwine}.
\end{remark}

\begin{example}
Consider the free product $\Gamma:=T_{1}\ast T_{1}\cong T_{2}$. Here, $\Aut(T_{1})\cong C_{2}$ and the two embedded automorphism groups of Proposition~\ref{prop:aut_intertwine} generate the infinite dihedral group $C_{2}\ast C_{2}\cong D_{\infty}=\Aut(T_{2})$. In particular, they do not commute.
\end{example}

Retain the notation of Proposition~\ref{prop:aut_intertwine}. Even if we choose groups $G_i\leq \Aut(\Gamma_i)$ $i\in\{1,\ldots,n\}$ with finitely many vertex-orbits on the respective $\Gamma_{i}$, the induced group $\smash{G:=\langle \widehat{G}_1, \dots, \widehat{G}_n\rangle\le\Aut(\Gamma)}$ need not have finitely many vertex-orbits on $\Gamma$: For example, let $\Gamma_{i}$ $(i\in\{1,\ldots,n\})$ be finite and choose $G_i$ to be trivial. Then the free product is infinite but $G$ is trivial. However, we have the following result.

\begin{corollary}\label{cor:vertex_transitive}
Let $\Gamma_{1},\ldots,\Gamma_{n}$ be connected, vertex-transitive graphs. Then the free product $\Conv_{i=1}^n \Gamma_i$ is vertex-transitive as well.
\end{corollary}
\begin{proof}
Set $\Gamma:=\Conv_{i=1}^{n}\Gamma_{i}$. Let $\tilde{v},\tilde{w}\in V(\Gamma)$ and let $\gamma=(e_{1},\ldots,e_{k})$ be a path from $\tilde{v}$ to $\tilde{w}$ with transition points $x_{1},\ldots,x_{l}\in V(\Gamma)$. By choosing the base points $x_{0}:=\tilde{v},x_{1},\ldots,x_{l},x_{l+1}:=\tilde{w}$ for $\Gamma$, Proposition~\ref{prop:aut_intertwine} and vertex-transitivity of the $\Gamma_{i}$ ($i\in\{1,\ldots,n\}$) yield automorphisms $\alpha_{1},\ldots,\alpha_{l+1}\in\Aut(\Gamma)$ such that $\alpha_{i}(x_{i-1})\!=\!x_{i}$ ($i\in\{1,\ldots,l+1\}$). Then $\alpha_{l+1}\circ\cdots\circ\alpha_{1}$ maps $\tilde{v}$ to $\tilde{w}$.
\end{proof}

The maps $\Aut(\Gamma_{j})\to\Aut(\Gamma)$ of Proposition~\ref{prop:aut_intertwine} are topological embeddings.

\begin{proposition}
Let $\Gamma_{1},\ldots,\Gamma_{n}$ be connected graphs and initialise $\Gamma:=\Conv_{i=1}^n \Gamma_i$ with $\mathbf{u}(\emptyset)=(u_1,\dots, u_n)\in\prod_{i=1}^{n}V(\Gamma_{i})$. For every $j\in\{1,\ldots,n\}$, the injective homomorphism $\varphi_{j}:\Aut(\Gamma_{j})\to\Aut(\Gamma)$, $\alpha\mapsto\hat{\alpha}$ is a topological embedding.
\end{proposition}

\begin{proof}
Consider the identity neighbourhood bases $U_{n}:=\Aut(\Gamma)_{B_{\Gamma}(\emptyset,n)}$ $(n\in\bbN)$ and $\smash{V_{n}:=\Aut(\Gamma_{j})_{B_{\Gamma_{j}}(u_{j},n)}}$ ($n\in\bbN$) for $\Aut(\Gamma)$ and $\Aut(\Gamma_{j})$ respectively. We have $\smash{\varphi_{j}^{-1}(U_{n})=V_{n}}$: On the one hand, $\smash{\varphi_{j}^{-1}(U_{n})\subseteq V_{n}}$ because $\smash{B_{S_{(\emptyset,j)}}(\emptyset,n)\subseteq B_{\Gamma}(\emptyset,n)}$ and $\varphi_{j}(\alpha)$ restricts to $\alpha$ on $\smash{S_{(\emptyset,j)}}$. Conversely, $\smash{\varphi_{j}^{-1}(U_{n})\supseteq V_{n}}$ by definition of $\varphi_{j}$ and because $d_{\Gamma_{j}}(u_{j}(x),u_{j})\le d_{\Gamma}(x,\emptyset)$ for all $x\in V(\Gamma)$. Also, $\varphi_{j}(V_{n})=U_{n}\cap\varphi_{j}(\Aut(\Gamma_{j}))$, hence $\varphi_{j}$ is a homeomorphism onto its image.
\end{proof}

Automorphisms of Proposition~\ref{prop:aut_intertwine} that fix $\emptyset\!\in\! V(\Gamma)$ are also covered by the~following proposition, whose greater flexibility entails that $\Aut(\Gamma)$ is often non-discrete. The notation~$\lambda(\tilde{v})$ for a vertex $\tilde{v}\!=\!v_{1}\ldots v_{l}$ in a free product $\Conv_{i=1}^{n}\Gamma_{i}$ refers to~the index $i\!\in\!\{1,\ldots,n\}$ such that $v_l\!\in\! V(\Gamma_i)$. We leave $\lambda(\emptyset)$ undefined and set $\{\lambda(\emptyset)\}\!:=\!\emptyset$.

\begin{proposition}\label{prop:aut_non_intertwine}
Let $\Gamma_{1},\ldots,\Gamma_{n}$ be connected graphs and initialise $\Gamma:=\Conv_{i=1}^{n}\Gamma_{i}$ with $\mathbf{u}(\emptyset)=(u_{1},\ldots,u_{n})\!\in\!\prod_{i=1}^{n}V(\Gamma_{i})$. For every $\tilde{v}\!\in\! V(\Gamma)$ and $i\!\in\!\{1,\dots, n\}\setminus\{\lambda(\tilde{v})\}$ let $\beta_{(\tilde{v},i)}\in \Aut(\Gamma_i)$. Recursively define $\beta$ on $V(\Gamma)$ by setting $\beta(\emptyset):=\emptyset$ and, for $\tilde{v}\in V(\Gamma)\setminus\{\emptyset\}$ and $v_{l+1}\in V(\Gamma_i)$ with $i\in\{1,\dots, n\}\setminus\{\lambda(\tilde{v})\}$, 
\begin{equation}
\label{def:beta}
\beta(\tilde{v}v_{l+1}) = \beta(\tilde{v})\beta_{(\tilde{v},i)}(v_{l+1}).
\end{equation}
Then $\beta\!\in\! \Aut(\Gamma)$ if $\beta_{(\tilde{v},i)}(u_i(\tilde{v}))\!=\!u_i(\beta(\tilde{v}))$ for all $\tilde{v}\!\in\! V(\Gamma)$ and $i\!\in\!\{1,\dots,n\}\setminus\{\lambda(\tilde{v})\}$.
\end{proposition}

\begin{proof}
First we show that $\beta(\tilde{v})\in V(\Gamma)$. Since $\beta_{(\tilde{v},i)}(v_{l+1})\in V(\Gamma_i)$, the first half of Equation~\eqref{eq:admissible} is satisfied. The condition that $v_{l+1}\ne u_i(\tilde{v})$ is preserved as well: Since $\beta_{(\tilde{v},i)}(u_i(\tilde{v})) = u_i(\beta(\tilde{v}))$ and $\beta_{(\tilde{v},i)}$ is a bijection, we have $\beta_{(\tilde{v},i)}(v_{l+1}) \ne  u_i(\beta(\tilde{v}))$. 

Next, we show that $\beta$ preserves all edge relations. Every edge $e$ of $\Gamma$ belongs to a unique $\Gamma_i$-sheet $S_{(\tilde{v},i)}$ for a unique~$i$. If $e$ is not incident on $\tilde{v}$, then $e = \{\tilde{v}v_{l+1}, \tilde{v}v'_{l+1}\}$ for some $\{v_{l+1}$, $v'_{l+1}\}\in E(\Gamma_i)$ and $\beta(e) = \{\beta(\tilde{v})\beta_{(\tilde{v},i)}(v_{l+1}), \beta(\tilde{v})\beta_{(\tilde{v},i)}(v'_{l+1})\}$ is an edge in the $\Gamma_i$-sheet $S_{(\beta(\tilde{v}),i)}$ because $\beta_{(\tilde{v},i)}$ is an automorphism of $\Gamma_i$. If $e$ is incident on $\tilde{v}$, then $e = \{\tilde{v},\tilde{v}v_{l+1}\}$ for some $v_{l+1}\in V(\Gamma_{i})$ such that $\{u_i(\tilde{v}),v_{l+1}\}$ in $E(\Gamma_i)$. Then $\beta(e) = \{\beta(\tilde{v}),\beta(\tilde{v})\beta_{(\tilde{v},i)}(v_{l+1})\}$ is an edge in the $\Gamma_i$-sheet $S_{(\beta(\tilde{v}),i)}$ because $\{u_i(\beta(\tilde{v})),\beta_{(\tilde{v},i)}(v_{l+1})\}\in E(\Gamma_{i})$.
\end{proof}

\begin{corollary}\label{cor:aut_non_discrete_stabiliser}
Let $\Gamma_{1},\ldots,\Gamma_{n}$ $(n\!\ge\! 2)$ be connected graphs and $\Gamma:=\Conv_{i=1}^{n}\Gamma_{i}$. If $\Aut(\Gamma_{j})_{u}\neq\{\id\}$ for some $j\!\in\!\{1,\ldots,n\}$ and $u\!\in\! V(\Gamma_{j})$ then $\Aut(\Gamma)$ is non-discrete.
\end{corollary}

\begin{proof}
Initalise $\Gamma$ with $\mathbf{u}(\emptyset)=(u_{1},\ldots,u_{n})\!\in\!\prod_{i=1}^{n}V(\Gamma_{i})$. By Corollary~\ref{cor:uniqueness} we may assume $u_{j}=u$. It suffices to show that the stabiliser $\Aut(\Gamma)_{B(\emptyset,k)}$ of the ball $B(\emptyset,k)$ of radius $k$ around $\emptyset\in V(\Gamma)$ is non-trivial for every $k\in\bbN$. Let $\beta_{0}\in\Aut(\Gamma_{j})_{u}$ be non-trivial. For every $\tilde{v}_{0}\in V(\Gamma)$ with $\lambda(\tilde{v}_{0})\neq j$ and $u_{j}(\tilde{v}_{0})=u$ define $\smash{\beta^{(\tilde{v}_{0})}}$ as in Proposition~\ref{prop:aut_non_intertwine} by setting
\begin{displaymath}
 \beta^{(\tilde{v}_{0})}_{(\tilde{v},i)}:=\begin{cases}\beta_{0} & \text{if $\tilde{v}_{0}$ is a prefix of $\tilde{v}$ and $i=j$} \\ \id & \text{otherwise}\end{cases}.
\end{displaymath}
Then $\beta^{(\tilde{v}_{0})}\in\Aut(\Gamma)$ by Proposition~\ref{prop:aut_non_intertwine}: To begin, we have
\begin{displaymath}
  \beta^{(\tilde{v}_{0})}_{(\tilde{v}_{0},j)}(u_{j}(\tilde{v}_{0}))=\beta_{0}(u)=u=u_{j}(\tilde{v}_{0})=u_{j}(\beta^{(\tilde{v}_{0})}(\tilde{v}_{0})).
\end{displaymath}
Further, if $\tilde{v}_{0}=v_{1}\ldots v_{l}$ is a prefix of $\tilde{v}=v_{1}\ldots v_{m}$ with $\lambda(\tilde{v})\neq j$, and the last letter from $\Gamma_{j}$ in the word $\tilde{v}$ has index $k>l$ then $\beta^{(\tilde{v}_{0})}_{(\tilde{v},j)}(u_{j}(\tilde{v}))=\beta_{0}(v_{k})=u_{j}(\beta^{(\tilde{v}_{0})}(\tilde{v}))$. In all other cases, we have $\beta^{(\tilde{v}_{0})}_{(\tilde{v},i)}(u_{i}(\tilde{v}))=u_{i}(\beta^{(\tilde{v}_{0})}(\tilde{v}))$ for all $i\in\{1,\ldots,n\}\backslash\lambda(\tilde{v})$ because neither $\beta^{(\tilde{v}_{0})}$ nor $\beta^{(\tilde{v}_{0})}_{(\tilde{v},i)}$ impact the $i$-th coordinate of $\mathbf{u}$. Since there are choices for $\tilde{v}_{0}$ with $d(\emptyset,\tilde{v}_{0})$ arbitrarily large, resulting in automorphisms $\beta^{(\tilde{v}_{0})}$ that stabilise $B(\emptyset,d(\emptyset,\tilde{v}_{0}))$ pointwise, the assertion follows
\end{proof}

In a similar fashion, we see that $\Aut(\Gamma)$ is also non-discrete when $\Gamma$ has at least three factors, two of which are isomorphic and vertex-transitive. In contrast to Corollary~\ref{cor:aut_non_discrete_stabiliser}, this applies to graphs with a regular automorphism group. Note that the assumption $n\ge 3$ cannot be weakened to $n\ge 2$ as the case $\Gamma:=T_{1}\ast T_{1}$ shows.

\begin{proposition}\label{prop:aut_non_discrete_isomorphic}
Let $\Gamma_{1},\ldots,\Gamma_{n}$ $(n\ge 3)$ be connected graphs and $\Gamma:=\Conv_{i=1}^{n}\Gamma_{i}$. If two factors of $\Gamma$ are isomorphic then $\Aut(\Gamma)$ is non-discrete.
\end{proposition}

\begin{proof}
We show that the stabiliser $\Aut(\Gamma)_{B(\emptyset,l)}$ of the ball $B(\emptyset,l)$ of radius $l$ around $\emptyset\in V(\Gamma)$ is non-trivial for every $l\in\bbN$. Let $\varphi:\Gamma_{i}\to\Gamma_{j}$ be an isomorphism for some distinct $i,j\in\{1,\ldots,n\}$ and fix $k\in\{1,\ldots,n\}\backslash\{i,j\}$. For every $\tilde{v}_{0}\in V(\Gamma)$ with $\lambda(\tilde{v}_{0})=k$ and $\varphi(u_{i}(\tilde{v}_{0}))=u_{j}(\tilde{v}_{0})$, define an automorphism $\smash{\beta^{(\tilde{v}_{0})}}\in\Aut(\Gamma)_{B(\emptyset,d(\emptyset,\tilde{v}_{0}))}$ as follows: Let $\tilde{v}=v_{1}\ldots v_{l}\in V(\Gamma)$. If $\tilde{v}_{0}$ is not a prefix of $\tilde{v}$, set $\beta^{(\tilde{v}_{0})}(\tilde{v})=\tilde{v}$. When $\tilde{v}_{0}$ is a prefix of $\tilde{v}$, we define $\beta^{(\tilde{v}_{0})}(\tilde{v})$ by induction on $r\ge 0$ where $\tilde{v}=\tilde{v}_{0}v_{m+1}\ldots v_{m+r}$, that is $l=l(\tilde{v})=l(\tilde{v}_{0})+r=m+r$. If $r=0$, put $\smash{\beta^{(\tilde{v}_{0})}(\tilde{v})=\tilde{v}}$. Now consider vertices of the form $\tilde{v}v_{m+(r+1)}$ where $l(\tilde{v})=m+r$. Set
\begin{displaymath}
 \beta^{(\tilde{v}_{0})}(\tilde{v}v_{m+(r+1)}):=\begin{cases}\beta^{(\tilde{v}_{0})}(\tilde{v})\varphi(v_{m+(r+1)}) & \text{if } v_{m+(r+1)}\in V(\Gamma_{i}) \\ \beta^{(\tilde{v}_{0})}(\tilde{v})\varphi^{-1}(v_{m+(r+1)}) & \text{if } v_{m+(r+1)}\in V(\Gamma_{j}) \\ \beta^{(\tilde{v}_{0})}(\tilde{v})v_{m+(r+1)} & \text{if } v_{m+(r+1)}\not\in V(\Gamma_{i})\sqcup V(\Gamma_{j})\end{cases},
\end{displaymath}
Then $\smash{\beta^{(\tilde{v}_{0})}}$ is a bijection on vertices, preserves both types of edges~\ref{eq:nonemptyneighbours1} and~\ref{eq:nonemptyneighbours2} and so does its inverse. Hence $\smash{\beta^{(\tilde{v}_{0})}}$ is an automorphism. Since there are choices for $\tilde{v}_{0}$ with $d(\emptyset,\tilde{v}_{0})$ arbitrarily large, resulting in automorphisms $\beta^{(\tilde{v}_{0})}$ that stabilise $B(\emptyset,d(\emptyset,\tilde{v}_{0}))$ pointwise, the assertion follows.
\end{proof}

As a consequence of Corollary~\ref{cor:aut_non_discrete_stabiliser} and Proposition~\ref{prop:aut_non_discrete_isomorphic}, the Cayley graph
\begin{displaymath}
\Gamma(\Conv_{i=1}^{n}G_{i},\sqcup_{i=1}^{n}S_{i})\cong\Conv_{i=1}^{n}\Gamma(G_{i},S_{i}) 
\end{displaymath}
of a free product of finitely generated groups $(G_{i},S_{i})$ $(i\in\{1,\ldots,n\})$ has non-discrete automorphism group whenever $\Aut(\Gamma(G_{i},S_{i}))\gneq G_{i}$ for some $i\in\{1,\ldots,n\}$, or $n\ge 3$ and $\Gamma(G_{i},S_{i})\cong \Gamma(G_{j},S_{j})$ for some distinct $i,j\in\{1,\ldots,n\}$.

The question of how $\Aut(\Gamma(G,S))$ relates to its subgroup $G$ is frequently studied, see {\it e.g.\/} \cite{W76}, \cite{DSV16} and \cite{LS18}. However, the focus often lies on minimising the index $[\Aut(\Gamma(G,S)):G]$. In the trivial case, the graph $\Gamma(G,S)$ is called a \emph{graphical regular representation} of $G$. For example, in the case of finite groups there is a complete classification of groups $G$ which admit a symmetric generating set $S$ such that $\Aut(\Gamma(G,S))=G$, see \cite{IW75} and \cite{G78}. From a t.d.l.c. group perspective we are interested in non-discrete automorphism groups of Cayley graphs and therefore ask the following.

\begin{question}\label{qu:cayley_graphs}
Does every finitely generated group $G$, other than $C_{2}$, admit a finite symmetric generating set $S$ such that $\Aut(\Gamma(G,S))\gneq G$?
\end{question}

Note that if a generator in $S$ is repeated then $\Aut(\Gamma(G,S))\gneq G$. Hence we suppose that the generators in Question~\ref{qu:cayley_graphs} are distinct.

In case of a positive answer to Question~\ref{qu:cayley_graphs}, we conclude that every free product of finitely generated groups admits a locally finite Cayley graph with non-discrete automorphism group: Free products other than $C_{2}\ast C_{2}$ are covered by the above discussion, whereas for $C_{2}\ast C_{2}\cong\langle a,b\mid a^{2},b^{2}\rangle$ the Cayley graph with respect to the finite symmetric generating set $\{a,a^{-1},b,b^{-1},ab,(ab)^{-1},bab,(bab)^{-1}\}$ has non-discrete automorphism group.

\section{Sheet-preserving automorphisms}\label{sec:aut_s}

Let $\Gamma_{1},\ldots,\Gamma_{n}$ be connected graphs and $\Gamma:=\Conv_{i=1}^{n}\Gamma_{i}$. We define two subgroups of $\Aut(\Gamma)$ that preserve the set of sheets $\calS:=\{S_{(v,i)}\mid v\in V(\Gamma),\ i\in\{1,\ldots,n\}\}$. Recall that the colour $c(S)\in\{1,\ldots,n\}$ of a sheet $S\in\calS$ is well-defined. We set
\begin{align*}
 \Aut_{\calS}(\Gamma)&:=\{\alpha\in\Aut(\Gamma)\mid \forall S\in\calS:\ \alpha(S)\in\calS\}, \\
 \Aut_{\calS,c}(\Gamma)&:=\{\alpha\in\Aut(\Gamma)\mid \forall S\in\calS:\ \alpha(S)\in\calS \text{ and } c(\alpha(S))=c(S)\}. 
\end{align*}
Note that $\Aut_{\calS,c}(\Gamma)\le\Aut_{\calS}(\Gamma)\le\Aut(\Gamma)$. If the factors of $\Gamma$ are $2$-connected then the set $\calS$ coincides with the maximal $2$-connected subgraphs of $\Gamma$, also known as ``blocks'', or ``lobes'', see e.g. \cite{JW77} and \cite{GW19}. In this case, $\Aut_{\calS}(\Gamma)=\Aut(\Gamma)$. If, in addition, the factors of $\Gamma$ are pairwise non-isomorphic we conclude that $\Aut_{\calS,c}(\Gamma)=\Aut_{\calS}(\Gamma)=\Aut(\Gamma)$.

\begin{example}\label{ex:sheet_subgroups}
The following examples illustrate the non-$2$-connected case.
\begin{enumerate}[(i)]
 \item Let $\Gamma:=T_{1}\ast T_{1}$. Then $\Gamma\cong T_{2}$ and $\Aut_{\calS}(\Gamma)=\Aut(\Gamma)=D_{\infty}$. Furthermore, $[\Aut_{\calS}(\Gamma):\Aut_{\calS,c}(\Gamma)]=[\Aut(\Gamma):\Aut_{\calS,c}(\Gamma)]=2$.
 \item\label{item:regular_tree} Let $\Gamma:=\Conv_{i=1}^{d}T_{1}\cong T_{d}$. Then $\Aut_{\calS}(\Gamma)=\Aut(\Gamma)$. Following the notation of Burger--Mozes \cite[Section 3.2]{BM00a} and Tornier \cite[Definition 3.1]{Tor20}, we have $\smash{\Aut_{\calS,c}(\Gamma)=\mathrm{U}_{1}^{(c)}(\{\id\})}$, which is discrete. In particular, we obtain $[\Aut_{\calS}(\Gamma):\Aut_{\calS,c}(\Gamma)]=[\Aut(\Gamma):\Aut_{\calS,c}(\Gamma)]=\infty$.
 \item\label{item:universal_transposition} Let $\Gamma:=T_{1}\ast T_{2}\cong T_{3}$. In this case, $\smash{\Aut_{\calS,c}(\Gamma)=\Aut_{\calS}(\Gamma)=\mathrm{U}_{1}^{(l)}(\langle\tau\rangle)}$ where $\tau:=(12)\in S_{3}$ and $l$ is the legal labelling of $T_{3}$ arising from the labellings $\pspicture(-0.1,-0.075)(1.1,0)\psdots(0,0)(1,0)\psline(0,0)(1,0)\rput(0.5,0.15){\scriptsize{$3$}}\endpspicture$ and $\pspicture(-0.6,-0.075)(3.7,0)\psdots(0,0)(1,0)(2,0)(3,0)\psline(0,0)(1,0)\psline(1,0)(2,0)\psline(2,0)(3,0)\uput[l](0,0){$\cdots$}\uput[r](3,0){$\cdots$}\rput(0.5,0.15){\scriptsize{$1$}}\rput(1.5,0.15){\scriptsize{$2$}}\rput(2.5,0.15){\scriptsize{$1$}}\endpspicture$ of the factors of $\Gamma$. We deduce $[\Aut(\Gamma)\!:\!\Aut_{\calS}(\Gamma)]\!=\!\infty$ as $\smash{[\Aut(\Gamma)_{\emptyset}\!:\!\Aut_{\calS}(\Gamma)_{\emptyset}]\!=\![\mathrm{U}_{1}^{(l)}(S_{3})_{\emptyset}\!:\!\mathrm{U}_{1}^{(l)}(\langle\tau\rangle)_{\emptyset}]\!=\!\infty}$.
\end{enumerate}
\end{example}

\begin{example}\label{ex:caprace_demedts}
In \cite[Example 3.7]{CM11}, Caprace and De Medts consider the free product $\Gamma:=T_{2}\ast T_{2}$ of two lines and the group $G:=\Aut_{\calS,c}(\Gamma)$.
\end{example}

Every graph that is not $2$-connected has a cut vertex. Therefore, the following proposition contrasts the case where all factors of $\Gamma$ are $2$-connected.

\begin{proposition}
Let $\Gamma_{1},\ldots,\Gamma_{n}$ be connected graphs and $\Gamma\!:=\!\Conv_{i=1}^{n}\Gamma_{i}$. If $\Gamma$ has two isomorphic factors which contain a cut vertex then $\Aut_{\calS}(\Gamma)\lneq\Aut(\Gamma)$. Conversely, if $\Aut_{\calS}(\Gamma)\!\lneq\!\Aut(\Gamma)$ then some factor of $\Gamma$ contains a cut vertex.
\end{proposition}

\begin{proof}
Let $\varphi:\Gamma_i\to\Gamma_j$ be an isomorphism for some distinct $i,j\in\{1,\ldots,n\}$. Pick a cut vertex $u_i\!\in\! V(\Gamma_i)$ and a decomposition $\smash{\Gamma_{i}\!=\!(\Gamma_{i}^{(1)},u_{i})\#(\Gamma_{i}^{(2)},u_{i})}$ of $\Gamma_{i}$ associated to it. Let $\smash{\Gamma_{j}=(\Gamma_{j}^{(1)},u_{j})\#(\Gamma_{j}^{(2)},u_{j})}$ be the image of the decomposition of $\Gamma_{i}$ under~$\varphi$. Making arbitrary choices of $u_{k}\in V(\Gamma_{k})$ ($k\in\{1,\ldots,n\}\backslash\{i,j\}$), we may assume without loss of generality that $\textbf{u}(\emptyset) = (u_{1},u_{2}, u_{3},\ldots,u_{n})$. Define an automorphism $\hat{\varphi}\in\Aut(\Gamma)\backslash\Aut_{\calS}(\Gamma)$ as follows. Let $\smash{B_{i}^{(1)},B_{j}^{(1)}\le\Gamma}$ be the subgraphs of $\Gamma$ spanned by the sets of vertices $\tilde{v}=v_{1}\ldots v_{l}$ with $\smash{v_{1}\in\Gamma_{i}^{(1)}}$ and $\smash{v_{1}\in\Gamma_{j}^{(1)}}$ respectively, and note that these subgraphs are two components of $\Gamma\setminus \{\emptyset\}$. Define
\begin{displaymath}
 \tilde{\varphi}:\bigsqcup_{k=1}^{n}V(\Gamma_{k})\to\bigsqcup_{k=1}^{n}V(\Gamma_{k}),\ v\mapsto\begin{cases} \varphi(v) & \text{if } v\in V(\Gamma_{i}) \\ \varphi^{-1}(v) & \text{if } v\in V(\Gamma_{j}) \\ v & \text{if } v\notin V(\Gamma_{i})\sqcup V(\Gamma_{j})\end{cases}.
\end{displaymath}
Now define $\hat{\varphi}(\emptyset)=\emptyset$ and, for $\tilde{v}=v_{1}\ldots v_{l}\in V(\Gamma)$, 
\begin{displaymath}
 \hat{\varphi}(v_{1}\ldots v_{l}):=\begin{cases}\tilde{\varphi}(v_{1})\ldots\tilde{\varphi}(v_{l}) & \text{if } \tilde{v}\in V(B_{i}^{(1)})\sqcup V(B_{j}^{(1)}) \\ v_{1}\ldots v_{l} & \text{if } \tilde{v}\not\in V(B_{i}^{(1)})\sqcup V(B_{j}^{(1)})\end{cases}.
\end{displaymath}
Then $\hat{\varphi}$ is the identity on all components of $\Gamma\setminus \{\emptyset\}$ except $B_{i}^{(1)}$ and $B_{j}^{(1)}$ and interchanges those components because it interchanges $i$ and $j$ in Equation~\eqref{eq:admissible}. It also preserves edges because $\varphi$ is an isomorphism. Hence, $\hat{\varphi}\in\Aut(\Gamma)$ but not in $\Aut_{\calS}(\Gamma)$ as it does not map the $\smash{\Gamma_{i}}$- and $\smash{\Gamma_{j}}$-sheets at $\emptyset\in V(\Gamma)$ to other sheets.

Now, assume $\Aut_{\calS}(\Gamma)\lneq\Aut(\Gamma)$. If no factor of $\Gamma$ contains a cut vertex, then all factos of $\Gamma$ are $2$-connected and hence $\Aut_{\calS}(\Gamma)=\Aut(\Gamma)$, as discussed above, in contradiction to the assumption.
\end{proof}

\subsection{Structure Tree}
To analyse $\Aut_{\calS}(\Gamma)$ further, we associate a \emph{structure tree} $T$ to $\Gamma$ which captures the fact that the graph $\Gamma$ is tree-like. Define
\begin{align*}
  V(T):=V(\Gamma)\sqcup\calS \text{ and } E(T):=\left\{\{v,S\}\mid v\in V(\Gamma),\ S\in\calS \text{ and } v\in V(S)\right\}.
\end{align*}
It is immediate from the definition that $T$ is a bipartite graph with parts $V(\Gamma)$ and~$\calS$. Then, since any two distinct vertices in $V(\Gamma)$ can belong to at most one sheet, and any two distinct sheets in $\calS$ intersect in at most one vertex, $T$ is a tree. When all factors of $\Gamma$ are $2$-connected the tree $T$ is precisely the block-cut tree of~$\Gamma$, see \cite[Section 4]{Har69}. Note that $T$ is locally finite if and only if all factors of $\Gamma$ are finite. Also, the actions of $\Aut_{\calS}(\Gamma)$ on the sets $V(\Gamma)$ and~$\calS$ induce an injective homomorphism $\tau:\Aut_{\calS}(\Gamma)\to\Aut(T)$.

\begin{remark}
The map $\tau:\Aut_{\calS}(\Gamma)\to\Aut(T)$ is in fact a topological embedding: To check continuity, let $F\subseteq V(T)$ be a finite set and decompose $F=F_{V}\sqcup F_{\calS}$ into vertices $F_{V}$ coming from vertices of $\Gamma$, and vertices $F_{\calS}$ coming from sheets of~$\Gamma$. For every $S\in F_{\calS}$, choose $e_{S}\in E(S)$. Then $\tau^{-1}(\Aut(T)_{F})$ contains the open set $\smash{\Aut_{\calS}(\Gamma)_{F_{V}\cup\{e_{S}\mid S\in F_{\calS}\}}}$. To see that $\tau$ is open onto its image, note that $V(\Gamma)\subseteq V(T)$ and therefore $\tau(\Aut(\Gamma)_{F})=\Aut(T)_{F}\cap\image\tau$ for any finite set $F\subseteq V(\Gamma)$.
\end{remark}

Using the standard metric $d'$ on $V(T)$, define a metric $d$ on $\calS\subseteq V(T)$ by $d:=\frac{1}{2}d'$, as well as the map $\Vert\cdot\Vert:\Aut_{\calS}(\Gamma)\to\bbN_{0},\ \alpha\mapsto\min\{d(S,\alpha(S))\mid S\in\calS\}$.

Let $\alpha\in\Aut_{\calS}(\Gamma)$. If $\Vert\alpha\Vert=0$ then $\alpha$ leaves a sheet invariant. For example, this applies to the automorphisms of Propositions~\ref{prop:aut_intertwine} and~\ref{prop:aut_non_intertwine}. The automorphisms in the following proposition satisfy either $\Vert\alpha\Vert=0$ or $\Vert\alpha\Vert=1$.

\begin{proposition}\label{prop:aut_sheet_swap}
Let $\Gamma_{1},\ldots,\Gamma_{n}$ be connected graphs, $\Gamma\!:=\!\Conv_{i=1}^{n}\Gamma_{i}$ and $\varphi:\Gamma_{i}\to\Gamma_{j}$ an isomorphism for distinct $i,j\in\{1,\ldots,n\}$. If $\tilde{v}_{0}\!\in\! V(\Gamma)$ satisfies $\varphi(u_{i}(\tilde{v}_{0}))\!=\!u_{j}(\tilde{v}_{0})$ then there is $\hat{\varphi}\in\Aut_{\calS}(\Gamma)_{\tilde{v}_{0}}$ with $\hat{\varphi}(S_{(\tilde{v}_{0},i)})=S_{(\tilde{v}_{0},j)}$ and $\hat{\varphi}(S_{(\tilde{v}_{0},j)})=S_{(\tilde{v}_{0},i)}$.
\end{proposition}

\begin{proof}
First, assume $\tilde{v}_{0}=\emptyset$ and set
\begin{displaymath}
 \tilde{\varphi}:\bigsqcup_{k=1}^{n}V(\Gamma_{k})\to\bigsqcup_{k=1}^{n}V(\Gamma_{k}),\ v\mapsto\begin{cases} \varphi(v) & \text{if } v\in V(\Gamma_{i}) \\ \varphi^{-1}(v) & \text{if } v\in V(\Gamma_{j}) \\ v & \text{if } v\notin V(\Gamma_{i})\sqcup V(\Gamma_{j})\end{cases}.
\end{displaymath}
Now, set $\hat{\varphi}(\emptyset)=\emptyset$. For $\tilde{v}=v_{1}\ldots v_{l}\in V(\Gamma)$, define
\begin{displaymath}
 \hat{\varphi}(v_{1}\ldots v_{l}):=\begin{cases}\tilde{\varphi}(v_{1})\ldots\tilde{\varphi}(v_{l}) & \text{if } v_{1}\in V(\Gamma_{i})\cup V(\Gamma_{j}) \\ v_{1}\ldots v_{l} & \text{otherwise}\end{cases}.
\end{displaymath}
Note that $\hat{\varphi}(S_{(\tilde{v}_{0},i)})=S_{(\tilde{v}_{0},j)}$ and $\hat{\varphi}(S_{(\tilde{v}_{0},j)})=S_{(\tilde{v}_{0},i)}$ by definition, that it maps sheets to sheets, and that it exchanges the roles of $i$ and $j$ in Equation~\eqref{eq:admissible}. It also preserves edge relations, because $\varphi$ is an automorphism, and $\hat{\varphi}(\emptyset)=\emptyset$. Hence $\hat{\varphi}$ belongs to $\Aut_{\calS}(\Gamma)_{\tilde{v}_{0}}$.

If $\tilde{v}_{0}\neq\emptyset$, consider the free product $\Gamma':=\Conv_{i=1}^{n}\Gamma_{i}$ initialised with $\mathbf{u}'(\emptyset)=\mathbf{u}(\tilde{v}_{0})$. Let $\psi:\Gamma'\to\Gamma$ be the isomorphism with $\psi(\emptyset)=\tilde{v}_{0}$ provided by Corollary~\ref{cor:uniqueness}. Then $\psi\circ\hat{\varphi}\circ\psi^{-1}$, where $\hat{\varphi}$ is applied to $\Gamma'$ serves.
\end{proof}

The following classification of sheet-preserving automorphisms of $\Gamma$ is analogous to \cite[Proposition 3.2]{Tit70}. A \emph{bi-infinite line of sheets} in $\Gamma$ is a map $\gamma:\bbZ\to\calS$ such that $d(\gamma(n),\gamma(m))=|n-m|$ for all $n,m\in\bbZ$. An automorphism $\alpha\in\Aut_{\calS}(\Gamma)$ acts as a \emph{translation of length $l$} along $\gamma$ if $\alpha(\gamma(n))=\gamma(n+l)$ for all $n\in\bbZ$.

\begin{theorem}\label{thm:aut_s_three_types}
Let $\Gamma_{1},\ldots,\Gamma_{n}$ be connected graphs, $\Gamma:=\Conv_{i=1}^{n}\Gamma_{i}$ and $\alpha\in\Aut_{\calS}(\Gamma)$. Then $\alpha$ satisfies exactly one of the following conditions. Either, it
\begin{enumerate}[leftmargin=2cm]
    \item[($\Vert\alpha\Vert=0$)] leaves a sheet invariant, or
    \item[($\Vert\alpha\Vert=1$)] fixes a vertex without leaving a sheet invariant, or
    \item[($\Vert\alpha\Vert\ge 1$)] is a translation of length $\Vert\alpha\Vert$ along a unique bi-infinite line of sheets.
\end{enumerate}
Furthermore, if $\Vert\alpha\Vert\le 1$ then $\alpha=\sigma\circ\beta$ for some automorphisms $\sigma,\beta\in\Aut_{\calS}(\Gamma)$ that leave a sheet invariant and stabilise a vertex respectively.
\end{theorem}

\begin{proof}
Consider the automorphism $\tau(\alpha)\in\Aut(T)$ induced by $\alpha\in\Aut_{\calS}(\Gamma)$. By \cite[Proposition 3.2]{Tit70}, the automorphism $\tau(\alpha)$ either fixes a vertex, inverts an edge or acts as a translation along a unique bi-infinite line in $T$. If $\tau(\alpha)$ fixes a vertex in $\calS\subseteq V(T)$ then $\alpha$ is of the first type. If $\tau(\alpha)$ fixes a vertex in $V(\Gamma)\subseteq V(T)$, but no vertex in $\calS\subseteq V(T)$, then $\alpha$ is of the second type. The case where $\tau(\alpha)$ is an inversion does not occur because $\alpha$ preserves the bipartition $V(T)=V(\Gamma)\sqcup\calS$. The final option is that $\tau(\alpha)$ acts as a translation of length $l$ along a unique bi-infinite line $\gamma':\bbZ\to V(T)$ in $T$. Without loss of generality, $\gamma'(0)\in\calS\subseteq V(T)$. Then $\alpha$ is a translation of length $l/2$ along $\gamma:\bbZ\to\calS$, where $\gamma(t)=\gamma'(2t)$, and $\Vert\alpha\Vert=l/2$.

Now consider the case $\Vert\alpha\Vert\le 1$. If $\Vert\alpha\Vert=0$ or $\alpha$ fixes a vertex then $\alpha$ decomposes trivially. We may therefore assume that $\alpha$ is a translation of length $1$ along a bi-infinite line $\gamma:\bbZ\to\calS$ of sheets. Let $i_{k}:=c(\gamma(k))\in\{1,\ldots,n\}$ for all $k\in\bbZ$. Since $\{1,\ldots,n\}$ is finite there are $k,l\in\bbZ$ with $k<l$ and $i_{k}=i_{l}$. Moreover, we may choose $l\in\bbZ_{>k}$ minimal with this property. Put $S:=\gamma(k)$. Note that the fact that $d(\gamma(k),\gamma(k+1))=1$ implies that $\gamma(k)\cap\gamma(k+1)=\{v\}$ for some $v\in V(\Gamma)$ and that $\gamma(k)$ and $\gamma(k+1)$ belong to sheets of different colours, so that $l\in\bbZ_{>k+1}$. Let $\mathbf{u}(v)=(v_{1},\ldots,v_{n})$. Then $u_{i_{k}}(\alpha^{l-k-1}(v))=u_{i_{k}}(v)=v_{i_{k}}$ as no minimal path from $v$ to $\alpha^{l-k-1}(v)$ passes through a sheet of type $i_{k}$. Hence the isomorphism $\smash{\varphi_{(v,i_{k+1})}^{-1}\circ\alpha^{-(l-k-1)}\circ\varphi_{(\alpha^{(l-k-1)}(v),i_{k})}:\Gamma_{i_{k}}\to\Gamma_{i_{k+1}}}$, where the maps $\varphi$ are the embeddings introduced in Example~ \ref{ex:mashup}\ref{item:mashup_free_product}, maps $v_{i_{k}}$ to $v_{i_{k+1}}$. Proposition~\ref{prop:aut_sheet_swap} now provides an automorphism $\beta\in\text{Aut}(\Gamma)$ which fixes $v\in V(\Gamma)$ and swaps $\gamma(k)$ and $\gamma(k+1)$. Then $\sigma:=\alpha\circ\beta$ is an automorphism of the free product graph that preserves the sheet $\alpha(S)$ and $\alpha=(\alpha\circ\beta)\circ\beta^{-1}$ as required.
\end{proof}

\begin{example}
Let $\Gamma:=T_{1}\ast T_{1}\ast T_{1}$. Then $\Gamma\cong T_{3}$, $\calS = E(T_{3})$, and $\Aut_{\calS}(\Gamma)=\Aut(\Gamma)$. Let $\alpha\in\Aut(T_{3})$. If $\alpha$ is elliptic then either $\alpha$ fixes an edge, in which case $\Vert\alpha\Vert=0$, or it does not fix any edge, in which case $\Vert\alpha\Vert=1$. If $\alpha$ is an inversion, then $\alpha$ leaves an edge and hence a sheet invariant and $\Vert\alpha\Vert=0$. (Note that inversions of $\Gamma$ are not inversions of $T$.) Finally, if $\alpha$ is hyperbolic of length $l$ then $\Vert\alpha\Vert=l$.
\end{example}

\subsection{Simplicity}

Here, we apply Tits' simplicity criterion for groups acting on trees to subgroups of $\Aut_{\calS}(\Gamma)$ via the injective homomorphism $\tau:\Aut_{\calS}(\Gamma)\to\Aut(T)$.

First, we briefly recall this criterion. See e.g. (\cite[Section 6.3]{GGT18}) for more detail. Let $C$ denote a finite or (bi)infinite path in $T$. For every $x\in V(C)$, the pointwise stabilizer $H_{C}$ induces an action $\smash{H_{C}^{(x)}\le\Aut(\pi^{-1}(x))}$ on the subtree spanned by those vertices of $T$ whose closest vertex in $C$ is~$x$; here, $\pi:V(T)\to V(C)$ denotes the closest point projection. We therefore obtain an injective homomorphism
\begin{displaymath}
 \varphi_{C}:H_{C}\to\prod\nolimits_{x\in V(C)}H_{C}^{(x)}.
\end{displaymath}
A subgroup $H\le\Aut(T)$ satisfies \emph{Property $P$} if $\varphi_{C}$ is an isomorphism for every path $C$ in $T$. Also recall that Tits defines the subgroup $G^+$ of $G$ to be that generated by the pointwise stabilisers of the edges of the tree $T$.

\begin{theorem*}[{\cite[4.5 Th{\'e}or{\`e}me]{Tit70}}]
Let $G\le\Aut(T)$. Suppose $G$ neither fixes an end nor preserves a proper subtree of $T$ setwise, and that $G$ satisfies Property $P$. Then the group $G^{+}$ is either trivial or simple.
\end{theorem*}

Now, let $H\le\Aut_{\calS}(\Gamma)$. Given a sheet $S\in\calS$ and $v\in V(S)$, we define the subgroup $H_{(S,v)}:=\{h\in H\mid h(v)=v \text{ and } h(S)=S\}$ and put
\begin{displaymath}
  H^{+_{\calS}}:=\langle\{H_{(S,v)}\mid S\in\calS,\ v\in V(S)\}\rangle (=\tau^{-1}(\tau(H)^{+})).
\end{displaymath}
If $\tau(H)\le\Aut(T)$ preserves no proper subtree, fixes no end of $T$, and satisfies Property (P) then $\tau(H)^{+}$ is either trivial or simple by \cite[4.5 Th{\'e}or{\`e}me]{Tit70}, and hence so is $H^{+_{\calS}}\cong\tau(H)^{+}$.

\begin{proposition}\label{prop:aut_s_simplicity}
Let $\Gamma_{1},\ldots,\Gamma_{n}$ be connected, vertex-transitive graphs and let $\Gamma:=\Conv_{i=1}^{n}\Gamma_{i}$. Then $\Aut_{\calS}(\Gamma)^{+_{\calS}}$ is either trivial or simple.
\end{proposition}

\begin{proof}
Let $G:=\Aut_{\calS}(\Gamma)$. We first show that $\tau(G)\le\Aut(T)$ is geometrically dense: Since the automorphisms of Corollary~\ref{cor:vertex_transitive} preserve the set of sheets $\calS$, the group $\tau(G)$ acts transitively on $V(\Gamma)\subseteq V(T)$. Hence it preserves no proper subtree of~$T$. Proposition \ref{prop:aut_intertwine} and vertex-transitivity of the factors $\Gamma_{i}$ imply that the stabiliser, $\tau(G)_{S}$, of a sheet $S\in\calS\subseteq V(T)$ acts transitively on $N_{T}(S)=V(S)\subseteq V(\Gamma)\subseteq V(T)$. In particular, none of the sheets neighbouring $S$ are fixed by $\tau(G)_{S}$. If $\tau(G)$ fixed an end of $T$, then $\tau(G)_{S}$ would have to fix every vertex on a ray beginning at~$S$, which would contradict that. Therefore, $\tau(G)$ fixes no end of $T$.

To see that $\tau(G)$ has Property (P), let $C$ be a finite or (bi)infinite path in $T$. Further, let $\pi:V(T)\to V(C)$ be the closest point projection. Given $\alpha\in G$ such that $\tau(\alpha)\in\tau(G)_{C}$ and $x\in V(C)$, consider the restriction $\tau(\alpha)_{x}$ of $\tau(\alpha)$ to $\pi^{-1}(x)$. Define $\alpha_{x}\in\Aut_{\calS}(\Gamma)$ by setting $\alpha_{x}(v):=\alpha(v)$ if either $v\in V(\pi^{-1}(C))$ or $v\in V(S)$ for some $S\in V(\pi^{-1}(C))$, and $\alpha_{x}(v):=v$ otherwise. Then $\tau(\alpha_{x})\in\tau(G)_{C}$, and we have $\tau(\alpha_{x})_{x}=\tau(\alpha)_{x}$ as well as $\tau(\alpha_{x})\in\mathrm{rist}(\pi^{-1}(C))$, the rigid stabiliser of $\pi^{-1}(C)$ introduced in Section~\ref{sec:preliminaries}, as required.
\end{proof}

\begin{example}
\hspace{0cm}
\begin{enumerate}[(i)]
 \item Let $\Gamma\!:=\!\Conv_{i=1}^{d}T_{1}\!\cong\! T_{d}$ be as in Example~\ref{ex:sheet_subgroups}\ref{item:regular_tree}. Then $\Aut_{\calS}(\Gamma)^{+_{\calS}}\!=\!\Aut(T_{d})^{+}$.
 \item Let $\Gamma:=T_{1}\ast T_{2}$ be as in Example~\ref{ex:sheet_subgroups}\ref{item:universal_transposition}. Then $\smash{\Aut_{\calS}(\Gamma)^{+_{\calS}}=\mathrm{U}_{1}(\langle\tau\rangle)^{+}}$.
\end{enumerate}
\end{example}

\subsection*{Acknowledgements}
We gratefully acknowledge research support by an AMSI Vacation Research Scholarship and the ARC Laureate Fellowship FL170100032. We owe thanks to Brian Alspach, Ben Brawn and Colin Reid for helpful discussions, and Ben Brawn for numerous improvements. Finally, the helpful comments of two anonymous reviewers were much appreciated.


\begin{thebibliography}{10}

\bibitem{BEW15}
C.~Banks, M.~Elder, and G.~A. Willis.
\newblock Simple groups of automorphisms of trees determined by their actions
  on finite subtrees.
\newblock {\it J. Group Theory}, 18(2) (2015), 235--261.

\bibitem{BW04}
U.~Baumgartner and G.~A. Willis.
\newblock Contraction groups and scales of automorphisms of totally
  disconnected locally compact groups.
\newblock {\it Israel J. Math.}, 142(1) (2004), 221--248.

\bibitem{BM00a}
M.~Burger and S.~Mozes.
\newblock Groups acting on trees: from local to global structure.
\newblock {\it Publ. Math. Inst. Hautes {\'E}tudes Sci.}, 92 (2000), 113--150.

\bibitem{BT19}
T.~P. Bywaters and S.~Tornier.
\newblock {Willis Theory via Graphs}.
\newblock {\it Groups, Geom. Dyn.}, (2019).

\bibitem{CM11}
P.-E. Caprace and T.~De~Medts.
\newblock Simple locally compact groups acting on trees and their germs of
  automorphisms.
\newblock {\it Transform. Groups}, 16(2) (2011), 375.

\bibitem{Cha07}
R.~Charney.
\newblock An introduction to right-angled artin groups.
\newblock {\it Geom. Dedicata}, 125(1) (2007), 141--158.

\bibitem{DSV16}
E.~Dobson, P.~Spiga, and G.~Verret.
\newblock Cayley graphs on abelian groups.
\newblock {\it Combinatorica}, 36(4) (2016), 371--393.

\bibitem{GGT18}
A.~Garrido, Y.~Glasner, and S.~Tornier.
\newblock {Automorphism Groups of Trees: Generalities and Prescribed Local
  Actions}.
\newblock In P.-E. Caprace and M.~Monod, editors, {\it New Directions in
  Locally Compact Groups}, Cambridge University Press, (2018), 92--116. 

\bibitem{Gle52}
A.~Gleason.
\newblock Groups without small subgroups.
\newblock {\it Annals of Math.}, (1952), 193--212.

\bibitem{G78}
C.D. Godsil.
\newblock {\it {GRR's for Non-solvable Groups}}.
\newblock Mathematics research report. University of Melbourne, (1978).

\bibitem{GW19}
J.~Graver and M.~Watkins.
\newblock Lobe, edge, and arc transitivity of graphs of connectivity.
\newblock {\it Ars Math. Contemp.}, 17(2) (2019), 581--589.

\bibitem{HIK11}
R.~Hammack, W.~Imrich, and S.~Klav{\v z}ar.
\newblock Handbook of graph products.
\newblock {\it Discrete Math. Appl.}, (2011).

\bibitem{Har69}
F.~Harary.
\newblock {\it Graph theory}.
\newblock Addison-Wesley, (1969).

\bibitem{IW75}
W.~Imrich and M.E. Watkins.
\newblock On automorphism groups of cayley graphs.
\newblock {\it Period. Math. Hungar.}, 7(3-4) (1976), 243--258.

\bibitem{JW77}
H.~Jung and M.~Watkins.
\newblock On the structure of infinite vertex-transitive graphs.
\newblock {\it Discrete Math.}, 18(1) (1977), 45--53.

\bibitem{KM08}
B.~Kr{\"o}n and R.~M{\"o}ller.
\newblock Analogues of cayley graphs for topological groups.
\newblock {\it Math. Z.}, 258(3) (2008), 637.

\bibitem{LS18}
P.-H. Leemann and M.~de~la Salle.
\newblock Cayley graphs with few automorphisms.
\newblock {\it J. Algebraic Combin.}, (2020), 1--30.

\bibitem{Moh06}
B.~Mohar.
\newblock Tree amalgamation of graphs and tessellations of the cantor sphere.
\newblock {\it J. Combin. Theory Ser. B}, 96(5) (2006), 740--753.

\bibitem{Moe02}
R.~M{\"o}ller.
\newblock Structure theory of totally disconnected locally compact groups via
  graphs and permutations.
\newblock {\it Canad. J. Math.}, 54(4) (2002), 795--827.

\bibitem{MSWZ18}
R.~M{\"o}ller, N.~Seifter, W.~Woess, and S.~Zemljic.
\newblock Amalagamated free products of graphs and arc-types.
\newblock {\it Unpublished}, (2018).

\bibitem{MZ52}
D.~Montgomery and L.~Zippin.
\newblock Small subgroups of finite-dimensional groups.
\newblock {\it Proc. Natl. Acad. Sci. USA}, 38(5) (1952), 440--442.

\bibitem{PT02}
T.~Pisanski and T.~Tucker.
\newblock Growth in products of graphs.
\newblock {\it Australas. J. Combin.}, 26 (2002), 155--170.

\bibitem{Que94}
G.~Quenell.
\newblock Combinatorics of free product graphs.
\newblock {\it Contemp. Math.}, 173 (1994), 257--257.

\bibitem{Tit70}
J.~Tits.
\newblock Sur le groupe des automorphismes d'un arbre.
\newblock In {\it Essays on topology and related topics}, Springer, (1970), 188--211.

\bibitem{Tor20}
S.~Tornier.
\newblock Groups acting on trees with prescribed local action.
\newblock arxiv preprint 2002.09876, (2020).

\bibitem{W76}
M.E. Watkins.
\newblock Graphical regular representations of free products of groups.
\newblock {\it J. Combin. Theory Ser. B}, 21(1) (1976), 47--56.

\bibitem{Wil94}
G.~A. Willis.
\newblock The structure of totally disconnected locally compact groups.
\newblock {\it Math. Ann.}, 300(1) (1994), 341--363.

\bibitem{Yam53}
H.~Yamabe.
\newblock {A generalization of a theorem of {G}leason}.
\newblock {\it Annals of Math.}, (1953), 351--365.

\end{thebibliography}
\end{document}